%% file: paper_v1.tex
\documentclass[a4paper,11pt]{article}
\usepackage[left=1.8cm,right=1.8cm,top=2.5cm,bottom=2.5cm]{geometry}
\usepackage[english]{babel}
\usepackage{authblk}
\usepackage{amssymb,amsmath,amsthm,mathscinet}
\usepackage[shortlabels]{enumitem}
\usepackage{bm}
\usepackage{graphicx,subcaption}
\usepackage{hyperref}
\usepackage{cleveref}
\usepackage{tikz,pgfplots}
\pgfplotsset{compat=1.15}
\usepackage{mathrsfs}
\usetikzlibrary{arrows}

\usepackage{color}
\usepackage{sidenotes}

\numberwithin{equation}{section}

\newcommand{\xx}{\mathbf{x}}
\newcommand{\ii}{\mathbf{i}}
\newcommand{\jj}{\mathbf{j}}
\newcommand{\uu}{\mathbf{u}}

\newcommand{\yy}{\mathbf{y}}
\newcommand{\YY}{\mathbf{Y}}

\newcommand{\der}{\textnormal{D}}

\newcommand{\mcA}{\mathcal{A}}

\newcommand{\mcI}{\mathcal{I}}
\newcommand{\mcL}{\mathcal{L}}
\newcommand{\mcN}{\mathcal{N}}
\newcommand{\mcU}{\mathcal{U}}
\newcommand{\Rset}{\mathbb{R}}
\newcommand{\Nset}{\mathbb{N}}
\newcommand{\Pprob}{\mathbb{P}}

\newcommand{\aalpha}{\boldsymbol{\alpha}}

\newtheorem{prop}{Proposition}[section]

\newtheorem{hp}{Assumption}[section]
\newtheorem{defn}{Definition}[section]
\newtheorem{remark}{Remark}[section]

\graphicspath{ {./figures_paper} }

\title{Uncertainty quantification analysis of bifurcations of the Allen--Cahn equation with random coefficients}	
\author[1,2]{Christian Kuehn}
\author[1]{Chiara Piazzola}
\author[1]{Elisabeth Ullmann}
\affil[1]{Department of Mathematics, TUM School of Computation, Information and Technology, Technical University of Munich, Germany}
\affil[2]{Munich Data Science Institute, Technical University of Munich, Germany}
\date{}

\begin{document}
	
\maketitle
	
\noindent
% Possible target journal: %\emph{Nonlinear dynamics (Springer)} \url{https://www.springer.com/journal/11071}
%\emph{IMA Journal of Applied Mathematics} \url{https://academic.oup.com/imamat}
% \emph{Mathematical Models and Methods in Applied sciences} \newline \url{https://www.worldscientific.com/toc/m3as/current}
%
%\emph{GAMM-Mitteilungen}\url{https://onlinelibrary.wiley.com/journal/15222608}
 
% \item \emph{Proceedings of ENUMATH 2023 (Springer)} \url{https://enumath2023.com/proceedings/}
% \item Special issue ``Numerical Approaches to Dynamical Systems'' with editors: Dimitri Breda, Cinzia Elia and Christian Pötzsche in \emph{Discrete and Continuous Dynamical Systems - Series B (American Institute of Mathematical Sciences)} \url{https://www.aimsciences.org/DCDS-B}
		
\section*{Abstract}
%Nonlinear partial differential equations are the ubiquitous tool for describing physical systems. Often, to account for modeling uncertainties, both deterministic parameters and random coefficients are included. 
%In this work we consider the Allen--Cahn equation, a prototypical problem in nonlinear dynamics that exhibits bifurcations in correspondence of variations of a deterministic parameter, and incorporate random spatially-heterogeneous effects by introducing a random field in the linear part of the equation. This results in the bifurcation points and bifurcation curves becoming random objects. We first perform a pathwise analysis to prove existence of pitchfork bifurcations. We make use of this information to propose an efficient use of the polynomial chaos expansion to tackle the forward uncertainty quantification analysis of bifurcations and bifurcation curves. 

In this work we consider the Allen--Cahn equation, a prototypical model problem in nonlinear dynamics that exhibits bifurcations corresponding to variations of a deterministic bifurcation parameter. Going beyond the state-of-the-art, we introduce a random coefficient in the linear reaction part of the equation, thereby accounting for random, spatially-heterogeneous effects. Importantly, we assume a spatially constant, deterministic mean value of the random coefficient. We show that this mean value is in fact a bifurcation parameter in the Allen--Cahn equation with random coefficients. Moreover, we show that the bifurcation points and bifurcation curves become random objects. We consider two distinct modelling situations: $(i)$ for a spatially homogeneous coefficient we derive analytical expressions for the distribution of the bifurcation points and show that the bifurcation curves are random shifts of a fixed reference curve; $(ii)$ for a spatially heterogeneous coefficient we employ a generalized polynomial chaos expansion to approximate the statistical properties of the random bifurcation points and bifurcation curves. We present numerical examples in 1D physical space, where we combine the popular software package Continuation Core and Toolboxes (COCO) for numerical continuation and the Sparse Grids Matlab Kit for the polynomial chaos expansion. Our exposition addresses both, dynamical systems and uncertainty quantification, highlighting how analytical and numerical tools from both areas can be combined efficiently for the challenging uncertainty quantification analysis of bifurcations in random differential equations.  
\bigskip

\noindent
\textbf{Keywords:} Nonlinear dynamical systems, random ordinary differential equations, dynamical systems, generalized polynomial chaos expansion, stochastic collocation, numerical continuation.

\bigskip

\noindent
\textbf{AMS Subject Classification:} 65P30, 65M99, 35R60, 37H20

\section{Introduction}
	
    The behavior of physical systems is formalized by the mathematical concept of dynamical systems. 
    This involves ordinary or partial differential equations (ODEs or PDEs, respectively) that describe the dependence of the state of the system on time. 
    In this work we are interested in the qualitative changes in the steady states of a dynamical system under variations of a model parameter.
    %that may be interpreted as a control parameter. 
    For example, it can happen that a stable state loses stability in correspondence of a critical value of such parameter, leading to the creation of a new stable state branching out from the previous one. 
    This is the case of the following ODE   
    \begin{equation}\label{eq:ode_pitchfork}
    \dot{x}(t) = px(t) - x(t)^3,
    \end{equation}
    where $x\colon \Rset \rightarrow \Rset$ is the state and $p \in \Rset$ is the deterministic parameter the impact of which we want to monitor. 
    The three steady state solutions are $x_1 = 0$ and $x_{2,3} = \pm \sqrt{p}$, with the latter existing for $p\geq 0$. 
    At $p=0$ the number of steady states changes and uniqueness of the solution is lost: there is one steady state for $p\leq 0$ and three steady states for $p> 0$. 
    The point $(p,x)=(0,0)$ and the parameter $p$  are usually called \textit{bifurcation point} and \textit{bifurcation parameter}, respectively. 
    Moreover, the stability of the trivial state $x_1$ changes at this point: it is globally linearly stable for $p<0$ and unstable for $p>0$. Instead, the nontrivial states $x_{2,3}$ that exist for $p>0$ are always stable. 
    This type of bifurcating behavior is called supercritical pitchfork bifurcation; we refer to \cite{strogatz2000,kielhofer2014} for an overview as well as for a broader introduction to the field of bifurcation theory.
    
    The starting point of our work is the Allen--Cahn equation \cite{henry1981,kuehn2019:book,robinson2001,uecker2017}
    \begin{equation}\label{eq:AC_general}
        %\begin{aligned}
            \partial_t u = \Delta u + p u - u^3, \quad p \in \Rset%, \\
        %               u & = 0,  \quad \quad & \xx \in \partial D,
        %\end{aligned}
    \end{equation}
    on an open bounded domain $D \subset \mathbb{R}^d$, $d=1,2,3$ with sufficiently smooth boundary $\partial D$, together with homogeneous Dirichlet boundary conditions. 
    The PDE in \eqref{eq:AC_general} is a reaction-diffusion equation, where the diffusion is modeled by the Laplace operator and the nonlinear reaction part has exactly the form of the ODE vector field in \eqref{eq:ode_pitchfork}. The Allen--Cahn equation is a prototypical model for bistability. It appears in a wide variety of applications with different names such as Nagumo equation in neuroscience, Chafee--Infante problem in nonlinear dynamics, real Ginzburg--Landau equation in statistical physics, $\Phi^4$ model in quantum field theory, population dynamics with Allee effect in ecology, or Schl\"ogl model in chemistry; the name Allen--Cahn originates in material science. This indicates already the broad applicability and importance of \eqref{eq:AC_general}. In addition, bistable reaction-diffusion equations have served as a major test problem for many mathematical techniques, e.g., in the theory of attractors, inertial manifolds and as normal forms. Similarly to \eqref{eq:ode_pitchfork}, the Allen--Cahn PDE exhibits supercritical pitchfork bifurcating behavior, see e.g.\ \cite{henry1981}. 
    To account for spatially-heterogeneous effects we replace the parameter $p$ in \eqref{eq:AC_general} with an $\xx$-dependent coefficient $q = q(\xx)$, $\xx \in D$, thereby following closely \cite{Kao2014}.
    % \textcolor{red}{To account for heterogeneous effects in the spatial domain we allow a spatially-dependent coefficient, now called $q = q(\xx)$, $\xx \in D$ instead of the homogeneous coefficient $p \in \Rset$, thereby following closely \cite{Kao2014}.} 
    We remark that this is only one way to embed spatial heterogeneity into \eqref{eq:AC_general}. 
    % \textcolor{red}{spatially-heterogeneous effects.} 
    Another possibility is the addition of forcing terms such as in \cite{Bastiaansen2022,Baer2002} also taking into account time-dependence as in e.g.~\cite{Hutt2008,Lefebvre2017}. The general question of how spatial heterogeneities affect the behavior of a dynamical system has been investigated in different contexts, e.g., for quantum systems~\cite{AizenmanWarzel}, for biological pattern formation~\cite{Page2003}, in climate dynamics~\cite{Bastiaansen2022}, and for chemistry~\cite{Baer2002}.
    In this work we go further and assume that the coefficient $q$ is a random field. That is, $q = q(\xx,\omega)$, with $\xx \in D$, $\omega \in \Omega$, and $\Omega$ indicating the set of elementary events in a probability space.
    This allows us to account for statistical variations and the lack of knowledge in the dynamical system, thereby incorporating information that is not captured by the PDE alone. 
    Hence, in this work we consider the following equation
    \begin{equation}\label{eq:AC_random_intro}
         \partial_t u = \Delta u + q(\xx,\omega) u - u^3.
    \end{equation}
    Importantly, we connect this model to the bifurcation theory for dynamical systems with deterministic coefficients by assuming that a deterministic hyper-parameter $p$ in the random field $q(\xx,\omega)$ induces bifurcations and is thus a bifurcation parameter.
    This novel viewpoint is a major contribution of our work. The bifurcation parameter is always denoted by $p$ in the following.
    
   Due to the presence of the random PDE coefficient we are confronted with a random dynamical system. Currently, the development of a mathematical theory and efficient computational methods for forward uncertainty propagation for bifurcation problems has just started to develop. For ODEs, the computation of random bifurcation points has been considered in~\cite{KuehnLux,Luxetal}, the study of random periodic orbits and invariant manifolds was addressed in~\cite{BredenKuehn1}, and the estimation of basins and attractors was discussed in \cite{Benedetti2023}. For PDEs, specific attempts were made, e.g., to derive the probability density function of time dependent solutions in the case of uncertain initial conditions \cite{Jornet2023}, for the computation of steady states of parabolic PDEs in \cite{Breden}, for bifurcations in fluid-dynamics problems in \cite{Lemaitre2004,sousedik2022,Venturi2010,gonnella2024stochastic}, and in the context of pattern formation in \cite{Bashkirtseva2023}. 
   %\textcolor{red}{For PDEs, the situation is even less developed, ~\cite{Bastiaansen2022,Breden,Jornet2023,Bashkirtseva2023,sousedik2022,gonnella2024stochastic}.} 
    
    The Allen--Cahn equation with random coefficients \eqref{eq:AC_random_intro} is an infinite-dimensional problem both in physical and probability space. In this work we assume a finite-dimensional parametrization of the probability space that results in approximating the random field $q(\xx,\omega)$ (and hence the solution of \eqref{eq:AC_random_intro}) by a finite number $N$ of random variables $Y_n(\omega)$, $n = 1, \ldots,N$. This is rather common to do, for example when employing a spectral decomposition of the random field such as the Karhunen--Lo\`eve expansion \cite{Karhunen1947, Loeve1978} truncated after a finite number of terms. Crucially, the first term of this expansion is the expected value of the random field. Moreover, we assume the expected value to be spatially-homogeneous and unknown and denote it by $p \in \Rset$ (hence being one of the hyper-parameters of the random field model). Thus, the random field $q$ can be expressed as $q(\xx,\omega) = p + g(\xx,\omega)$, where $g$ is a random field with expected value equal to zero. Then the problem \eqref{eq:AC_random_intro} turns into a PDE that has the same structure as the classic Allen--Cahn equation \eqref{eq:AC_general} only with an additional linear term with random and spatially-heterogeneous coefficient $g(\xx,\omega)$. 
    %\textcolor{red}{The arising PDE has the same structure of the classic Allen--Cahn equation \eqref{eq:AC_general} only with an additional linear term with random and spatially-heterogeneous coefficient $g(\xx,\omega) = q(\xx,\omega) - p$.} 
    From this we conclude that the expected value $p$ of the random field $q$ can be interpreted as a bifurcation parameter for the Allen--Cahn equation with random coefficients \eqref{eq:AC_random_intro}. 
    
    The goal of this work is to contribute to the development of a theoretical framework and suitable numerical methodologies for the forward Uncertainty Quantification (UQ) analysis of nonlinear dynamics focusing on bifurcations.
    The key observation is that the finite-dimensional representation of the random field $q$ in terms of $N$ statistically independent random variables allows us to consider the Allen--Cahn equation with random coefficients \eqref{eq:AC_random_intro} on a finite-dimensional probability space: The set of elementary events is the Cartesian product of the image spaces of the random variables $Y_n$, and the probability measure is absolutely continuous with respect to the Lebesgue measure on $\mathbb{R}^N$, where the associated probability density function is the product of the probability density functions of the random variables $Y_n$, see e.g.\ \cite{babuska2010}. Crucially, the bifurcation analysis of such a problem can be reduced to analysing the deterministic instances corresponding to realizations of the random variables.
    This resembles the classical bifurcation analysis for the Allen--Cahn equation with deterministic coefficients and lays the foundation for the bifurcation analysis of the random case.
    Specifically, we prove that the Allen--Cahn equation with random coefficients exhibits supercritical pitchfork bifurcations, where the bifurcation points are random variables and the bifurcation branches are random curves. 
    Moreover, we show that the UQ analysis in the simpler case of a spatially-homogeneous random coefficient $q = q(\omega)$ can be carried out almost analytically, i.e.\ without introducing any UQ numerical methods which are instead required to handle the general case of spatially-heterogeneous coefficients. 
    
    Thus, we propose a methodology based on the Polynomial Chaos (PC) expansion, a classical method in UQ, to construct surrogate models to give a probabilistic description of bifurcation points and steady states. The PC method was first proposed in \cite{wiener1938} to approximate random functions depending on Gaussian random variables. 
    Its extension to non-Gaussian random variables was proposed in \cite{XiuKarniadakins2002:gPC} with the name of generalized PC (gPC), see \cite{XiuKarniadakins2002:gPC,ernst2012}. 
    Our choice of using the PC method is motivated by two main reasons. 
    First, the PC method has already been exploited e.g.~in \cite{sousedik2022} and very recently in \cite{gonnella2024stochastic} in the context of bifurcation analysis of Navier--Stokes equations with stochastic viscosity.
    In particular, in \cite{sousedik2022} the authors use the PC expansion as a surrogate to monitor the sign of the real part of the rightmost eigenvalue of the linearization operator to assess linear stability. 
    In other words, the PC expansion is employed as surrogate for the solution of a stochastic eigenvalue problem. Secondly, the PC method can be implemented in a fully non-intrusive way, i.e.\ relying on the solution of deterministic problems that are realizations of the original random one. Such realizations are prescribed by an appropriate stochastic collocation strategy, that we choose to base on sparse grids, see e.g.\ \cite{bungartz2004,babuska2010,piazzola2022:sparse.grids,xiu2005}. This allows to take advantage of the well-established numerical methods for deterministic dynamical systems employing them in an (almost) black-box manner. In particular, well-studied numerical continuation methods (see e.g.\ \cite{uecker2021}) are used for the approximation of the branches of equilibria. 
    %Furthermore, note that the finite-dimensional framework discussed above provides the parametric regularity that can be exploited to study the approximation properties of this approach; we leave this investigation to future studies.   
    The study of the parametric regularity of the Allen--Cahn PDE solution together with approximation properties of the PC approach are left for future investigations.
    
    The rest of the paper is organized as follows. In Section~\ref{sect:AC} we summarize the main elements of the theory of bifurcations for the classic Allen--Cahn equation with deterministic coefficients. 
    Whereas the technical setting for bifurcation problems is standard in the dynamical systems community, we believe that it is not well-established in the UQ field yet and therefore worth recalling it. This also allows us to fix notation and terminology. 
    Experts in the bifurcation theory of dynamical systems may skip Section~\ref{sect:AC} and go directly to Section~\ref{sect:AC_random}, where we introduce the Allen--Cahn equation with random coefficients and specify the finite-dimensional random field model. 
    Furthermore, in Section~\ref{sect:AC_random} we detail the bifurcation analysis and prove general statements on bifurcation points and bifurcation curves.     
    The UQ analysis is first discussed in Section \ref{sect:AC_hom_rand} for the case of spatially-homogeneous random coefficients. In this setting, we obtain some analytic results that ease substantially the UQ analysis. The case of spatially-heterogeneous coefficients is further discussed in Section \ref{sect:AC_het_rand}, where we introduce our proposed numerical method based on the PC expansion and numerical continuation. Finally, in Section \ref{sect:conclusion} we summarize and discuss our approach, and give an outlook on future work. The two appendices, Appendix \ref{appendix:num_cont} and \ref{appendix:sg}, recall briefly the basis of numerical continuation and sparse grids, the backbones of the proposed numerical methodology.

\section{The Allen-Cahn equation with deterministic coefficients} \label{sect:AC}

    We consider the Allen--Cahn equation on an open bounded domain $D \subset \Rset^d$, $d=1,2,3$ with smooth boundary $\partial D$, i.e.,\
    	\begin{equation}\label{eq:AC}
    		\begin{aligned}
    			\partial_t u & = \Delta u + pu - u^3, \quad \quad  & \xx \in D, \\
    			u & = 0,  \quad \quad & \xx \in \partial D, \\
                u \vert_{t=0} & = u_0,  \quad &  \xx \in \overline D,
    		\end{aligned}
    	\end{equation}
    where $u = u(t,\xx): \Rset_+ \times \overline D  \rightarrow \Rset$ and $ p \in \Rset $ is a deterministic parameter. 
    The problem is well-posed in $L^2(D;\Rset)$ %$$H^1_0(D;\Rset)$ 
    and defines a continuous-time semiflow $\varphi^t$ on the time domain $\mathcal{T}=[0,+\infty)$, with state space $\mathcal{X}=L^2(D;\Rset)$, %H^1_0(D;\Rset)$ 
    and the family of evolution operators is given by 
   	\[
	% \varphi^t: H^1_0(D;\Rset) \rightarrow H^1_0(D;\Rset), \quad \varphi^t(u_0) = u(t,\xx). 
    \varphi^t: L^2(D;\Rset) \rightarrow L^2(D;\Rset), \quad \varphi^t(u_0) = u(t,\xx). 
	\]
    % where $u(t,\xx)$ is the solution of \eqref{eq:AC}, see e.g.\ \cite{robinson2001}. 
    In fact, the negative sign in front of the cubic nonlinearity guarantees that~\eqref{eq:AC} is globally dissipative, so that solutions exist for all times~\cite{robinson2001,temam1997}. 
    For the understanding of the long-term dynamics it is crucial to study the so-called steady states (also known as equilibria), that is, solutions of \eqref{eq:AC} which are forward invariant and independent of time.
    In particular, we classify the dynamics of \eqref{eq:AC} depending on the parameter $p$. 
    In the following we summarize the main elements of the theory, thereby fixing the notation. 
    
    \subsection{Bifurcation analysis} \label{sect:bif_analysis_det}
    
    Equilibria $u = u(\xx)$ are solutions of the stationary problem associated to \eqref{eq:AC}, i.e.
	\begin{equation}\label{eq:AC_stationary}
		\begin{aligned}
			\Delta u + p u  - u^3 & = 0, \quad \quad  & \xx \in D, \\
			u & = 0,  \quad \quad & \xx \in \partial D. \\
		\end{aligned}
	\end{equation}
    More precisely, since in general different values of $p$ lead to different solutions of \eqref{eq:AC_stationary}, we call an \textit{equilibrium point} each couple $(p,u)  \in \Rset \times L^2(D;\Rset)$ that satisfies \eqref{eq:AC_stationary}. Evidently, the trivial function $u\equiv 0$ is a solution of \eqref{eq:AC_stationary} for any $p \in \Rset$ and the set 
	\[
	   \{ (p,0) \, \vert \, p\in \Rset\} \subset \Rset \times L^2(D;\Rset)
	\]
    is then called \textit{trivial equilibrium branch}. Note that we use the term \textit{branch} for a connected set of points. The existence of other equilibrium branches intersecting the trivial one can be proven by the Crandall--Rabinowitz theorem, see e.g.\ \cite{kielhofer2014,kuehn2019:book,uecker2021}. Let us introduce the operator 
	\begin{equation}\label{eq:F}
	   F: \Rset \times L^2(D;\Rset) \rightarrow L^2(D;\Rset), %L^2(D;\Rset), 
        \quad F(p,u): = \Delta u + p u  - u ^3,
	\end{equation}
    % with open sets $P \subset \Rset$, $U \subset L^2(D;\Rset)$ %$\subset L^2(D;\Rset)$ %H^1_0(D;\Rset)$ 
    %open sets, 
    such that problem \eqref{eq:AC_stationary} can be rewritten as 
	\begin{equation}\label{eq:equilibrium_problem}
		F(p,u) = 0. 
	\end{equation}
    Further, let $(\der_u F)(p,0)$ denote the Frech\' et derivative of $F$ with respect to $u$ at a point $(p,0)$ on the trivial equilibrium branch. The candidate bifurcation points $(p^*,0)$ are such that $(\der_u F)(p^*,0)$ is not bijective, i.e.\ points where the Implicit Function theorem cannot be applied. 
    Hence, one has to look for the zero eigenvalues of the following $p$-dependent operator
     \begin{equation}\label{eq:frechet_der}
	(\der_u F) (p,0): L^2(D;\Rset) \rightarrow L^2(D;\Rset), \quad  [(\der_u F) (p,0)] v = (\Delta + p) v
	\end{equation}
    with domain $H^2(D;\Rset) \cap H^1_0(D;\Rset)$. 
    It has a self-adjoint and compact resolvent (see e.g.\ \cite{kielhofer2014}), thus its spectrum is discrete and consists of a non-increasing sequence of real eigenvalues $\lambda^F_i = \lambda^F_i(p)$, $ i \in \Nset$, $\lambda^F_i \rightarrow -\infty$, $i\rightarrow +\infty$. 
    %\textcolor{red}{It is a compact and self-adjoint operator (see e.g.\ \cite{kielhofer2014}). 
    %Thus
    %\begin{equation}\label{eq:frechet_der}
	%(\der_u F) (p,0): H^2(D;\Rset) \cap H^1_0(D;\Rset) \rightarrow L^2(D;\Rset), \quad  %(\der_u F) (p,0) v = (\Delta + p) v, 
	%\end{equation}
    %which is self-adjoint. 
    %its spectrum is discrete and consists of a non-increasing sequence of real eigenvalues $\lambda^F_i = \lambda^F_i(p)$, $ i \in \Nset$, $\lambda^F_i \rightarrow -\infty$, $i\rightarrow +\infty$.}
    Note that $\lambda^F_i = \lambda_i + p$, $i \in \Nset$, where  $\lambda_i$ denotes an eigenvalue of the Laplacian on $D$ with homogeneous Dirichlet boundary conditions. Hence, we conclude that there exists a sequence $(p^\ast_i )_{i \in \Nset}$ of bifurcation points with $ p^*_i = -\lambda_i$. 
    
    % For simplicity, let us drop the subscript $i$ and denote a generic bifurcation point by $p^*$.    
    The crucial assumption of the Crandall--Rabinowitz theorem is     
	\begin{equation}\label{eq:hp_CR}
	   \text{dim } \mathcal{N} \left[ (\der_u F)(p^*_i,0) \right] = \text{codim } \mathcal{R} \left[ (\der_u F)(p^*_i,0) \right]  = 1,
	\end{equation}
    where $\mathcal{N} \left[ (\der_u F)(p^*_i,0) \right]$ and $\mathcal{R} \left[ (\der_u F)(p^*_i,0) \right]$ are the kernel and the range of $(\der_u F)(p^*_i,0)$, respectively. 
    \begin{remark}The Crandall--Rabinowitz theorem handles the case of bifurcations from simple eigenvalues since the kernel of $(\der_u F)(p^*_i,0)$ is the eigenspace relative to the eigenvalue $0$. 
    In 1D physical space the condition \eqref{eq:hp_CR} is always satisfied. 
    In contrast, in higher space dimensions we can handle the first bifurcation point only, i.e.\ the one that corresponds to the rightmost eigenvalue of the Laplacian.
    This is because the other eigenvalues are in general not simple as the multiplicity of the eigenvalues of the Laplacian relates to the shape of the domain $D$, see e.g.~\cite{henrot2006:eigenvalues,Grebenkov2013:eigenvalues}.
    \end{remark} \noindent
    Crucially, condition \eqref{eq:hp_CR} allows us to consider the following decompositions 
	\begin{align*}
	   H^2(D;\Rset) \cap H^1_0(D;\Rset) & = \mathcal{N} \left[ (\der_u F)(p^*_i,0) \right] \oplus X_0, \\ \nonumber
	   L^2(D;\Rset) &  = \mathcal{R} \left[ (\der_u F)(p^*_i,0) \right] \oplus Z_0,
	\end{align*}
    with $X_0$ and $Z_0$ denoting the complements of $\mathcal{N} \left[ (\der_u F)(p^*_i,0) \right]$ and $\mathcal{R} \left[ (\der_u F)(p^*_i,0) \right]$, respectively, and $\text{dim } Z_0 = 1$.
    Then, by means of the so-called Lyapunov--Schmidt reduction method (see e.g.\ \cite{kielhofer2014}), we construct the following finite-dimensional projection of the infinite-dimensional problem \eqref{eq:equilibrium_problem},  % is (locally) equivalent to a one-dimensional problem 
	\begin{equation} \label{eq:finite_dim_bif}
	   \Phi(p,w) = 0, \quad \text{with } \, \Phi: \widetilde{P}  \times \widetilde{U} \rightarrow Z_0 \, \text{ and } \, \Phi(p^*_i,0) = 0,
	\end{equation}
    where $\widetilde{P}  \times \widetilde{U} \subset \Rset \times \mathcal{N} \left[ (\der_u F)(p^*_i,0) \right]$ denotes a neighborhood of $(p^*_i,0)$. 
    With the additional assumption that $F \in C^3(\Rset \times  L^2(D;\Rset); L^2(D;\Rset))$ and the so-called transversality condition
    \[
	\der_{u p} \, F(p^*_i,0) \, v^*_i \notin \mathcal{R} \left[ (\der_u F)(p^*_i,0) \right], 
	\]
    where $\mathcal{N} \left[ (\der_u F)(p^*_i,0) \right] = \text{span} \{v^*_i\}$, $\lVert v^*_i \rVert_{H^2} = 1$, we can apply the Implicit Function theorem to \eqref{eq:finite_dim_bif}. 
    The conclusion is that the nontrivial equilibrium branch through the point $(p^*_i,0)$, for which \eqref{eq:hp_CR} holds, can be described by a $C^1$ curve in $\Rset \times L^2(D;\Rset)$, %H^1_0(D;\Rset)$, 
	\begin{equation}\label{eq:branch}
		\{ \left(p(s),  u(s) \right) \, \lvert \, s \in (-\delta, \delta ), \ (p(0),u(0)) = (p^*_i,0)\}
	\end{equation}
    such that 
    \begin{equation}\label{eq:branch_formula}
    \left(p(s),  u(s) \right) = \left( \varphi(s), s v^*_i + \mathcal{O}(s^2) \right)
    \end{equation}
    as $s \rightarrow 0$, where $\varphi: (-\delta,\delta) \rightarrow \widetilde{P}$ is a $C^1$ function, $s \in \Rset$, and $\delta >0$ is a suitable real-valued constant. 
    We refer to the proof of the Crandall--Rabinowitz theorem in e.g.~\cite{kielhofer2014} for the derivation of this result. Further, the local shape of a nontrivial branch is determined by the derivatives with respect to $s$ of the nontrivial branch evaluated at $s=0$ (i.e.\ at the bifurcation point). The first derivative is 
	\begin{equation}\label{eq:tang_vect_branch}
	(\dot{p}(0), \dot{u}(0)) := \frac{\mathrm{d}}{\mathrm{d}s} \left(p(s),u(s)\right)\vert_{s=0} =  (0,v^*_i)
	\end{equation}
    with $\dot{p}(0)=0$ following from $(D_{uu} F) (p^*_i,0)[v^*_i,v^*_i] \in \mathcal{R}\left[ (\der_u F)(p^*_i,0) \right]$ (see \cite{kielhofer2014}). Instead
    $\dot{u}(0)=v^*_i$ follows from \eqref{eq:branch_formula}. 
    Thus, the branch of equilibria bifurcates from the trivial equilibrium in the direction of the kernel of $(\der_u F)(p^*_i,0)$. A more detailed shape of the bifurcation branch locally near $(p^*_i,0)$ can be determined via the second-order derivative
    %; this requires the third-order derivative of $\Phi$ 
    which involves the third-order derivative of $F$. Clearly, the relevant term is the cubic one and its sign (in this case negative) allows us to conclude that the bifurcation is a supercritical pitchfork, see \cite{kielhofer2014,uecker2021} for more details.  
    Note that pitchfork bifurcations are expected due to the symmetry $u\mapsto -u$ of the Allen--Cahn equation, that is, for each solution $u$ the function $-u$ is also a solution.  

    Note that the Crandall--Rabinowitz theorem provides only a local existence result. 
    However, under suitable regularity assumptions bifurcation branches can be extended globally, i.e., either a branch extends to be unbounded or it intersects other branches~\cite{kielhofer2014,uecker2021}.     
    Due to the just mentioned symmetry, 
    %For example, \eqref{eq:AC_stationary} has a symmetry implying that $u$ and $-u$ are solutions. Therefore 
    we can split each nontrivial branch emanating at a pitchfork bifurcation point into two parts and consider the two symmetric branches $\gamma^+_i \subset \Rset \times L^2(D;\Rset)$ and $\gamma^-_i \subset \Rset \times L^2(D;\Rset)$
    %$\gamma^+_i \subset \Rset \times H^1_0(D;\Rset_+)$ and $\gamma^-_i \subset \Rset \times H^1_0(D;\Rset_-)$ through the same point $(p^*_i,0)$ defined as follows
	\begin{align} \label{eq:symmetric_branch}
		\begin{split}
			\gamma^+_i & = \left\{ \left(p(s),  u(s) \right) \, \lvert \, s \in [0,s_c^+], \ (p(0),u(0)) = (p^*_i,0) \right\}, \\ 
			\gamma^-_i & = \left\{ \left(p(s),  u(s) \right) \, \lvert \, s \in {[0,s_c^-]}, \ (p(0),u(0)) = (p^*_i,0) \right\},
		\end{split}
	\end{align}
    where $s_c^+,s_c^- \in \Rset$ are, respectively, the right and left endpoints of the curve parametrization interval. In other words, since, in general, each bifurcation curve intersects other branches, we define a bifurcation branch as a curve connecting two intersection points (i.e.\ bifurcation points) arising at $s=0$ and $s=s_c^+$ (and similarly for $s_c^-$). 
    It is extremely crucial to note that for general nonlinear differential equations, as well for \eqref{eq:AC} in $d\geq2$, it is usually not possible to track the global shape of the a bifurcation branch, or even the spatial shape of the steady states by analytical pen-and-paper methods.  
    An exception is the Allen--Cahn equation in 1D space ($d=1$), where the bifurcation branches $\gamma_i^+$ and $\gamma_i^-$ can be globally extended such that $s_c^+ = +\infty$ and $s_c^- = + \infty$. This follows from a result by Chafee and Infante \cite{chafee1974} (see also \cite{henry1981,robinson2001}) showing that the only nontrivial equilibria are the ones bifurcating from the trivial one in correspondence of the eigenvalues of the Laplacian.
    Hence no additional bifurcation points arise along the nontrivial branches. 
      
    To complete the qualitative analysis of the Allen--Cahn equation \eqref{eq:AC} we now consider the stability of the equilibria. 
    A generic equilibrium point $(p^*,u^*)$ is called \textit{linearly stable} if the spectrum of $(\der_u F) (p^*,u^*)$ is properly contained in the left half of the complex plane. 
    Following up on the previous discussion, we conclude that the trivial equilibrium branch is stable for values of $p$ smaller than the first bifurcation point $p^*_1$, i.e.\ $p < p_1^*$. For $p>p^*_1$ the spectrum contains at least one positive eigenvalue and hence the equilibrium is unstable. Moreover, one can also pass from linear stability analysis to local nonlinear stability in our context by standard methods~\cite{DaleckiiKrein}. So one concludes that (nonlinear) local asymptotic stability holds for the trivial branch before the first bifurcation point. In addition, by the so-called \textit{principle of exchange of stability} (see e.g.\ \cite{kielhofer2014}) the first nontrivial equilibrium branch is locally asymptotically stable for a supercritical pitchfork bifurcation. Indeed, the stability of the trivial equilibrium branch for $p<p_1^*$ is ``transferred'' to the nontrivial one at the value $p = p_1^*$. On the other hand, the other nontrivial equilibrium branches arising at $p = p_i^*$, $i\geq2$ are unstable, as they bifurcate from a locally fully unstable branch near $p = p_i^*$, $i\geq2$. 

    Finally, we discuss the \textit{bifurcation diagram}, a classic tool to visually summarize the qualitative behavior of a dynamical system.
    In a bifurcation diagram we display equilibria and their stability against the bifurcation parameter $p$.
    Since we work in an infinite-dimensional state space we introduce scalar-valued observables associated to the $u$-component of a branch \eqref{eq:symmetric_branch}.
    This allows us to show 2D plots.
    Some examples are the $L^2(D)$-norm of $u$, or the value of $u$ at a point of interest of the domain $D$. 
    % In UQ these observables are often called \textit{quantities of interest (QoI)}.
    In Figure \ref{fig:AC_bif_diag} we display the $L^2(D)$-bifurcation diagram of the Allen--Cahn equation \eqref{eq:AC} on the one-dimensional physical domain $D=[0,\pi]$. 
    Note that in this case all eigenvalues $\lambda^F_i$ of $(\der_u F)(p,0)$ are simple.
    Thus nontrivial supercritical pitchfork bifurcation branches arise in correspondence of each of them. Stable parts of a branch are drawn with a solid line; dashed lines depict unstable parts of a branch. Note that the bifurcation diagram has been obtained with the numerical continuation software Continuation Core and Toolboxes (COCO) \cite{dankowicz2013} (version October 26, 2023) available at \cite{coco:sourceforge} after employing a finite difference discretization of the differential operator (the set-up is the same as in Section~\ref{sect:UQ_hom_rand}).

    \begin{figure}
		\centering
			\includegraphics[width=0.35\textwidth]{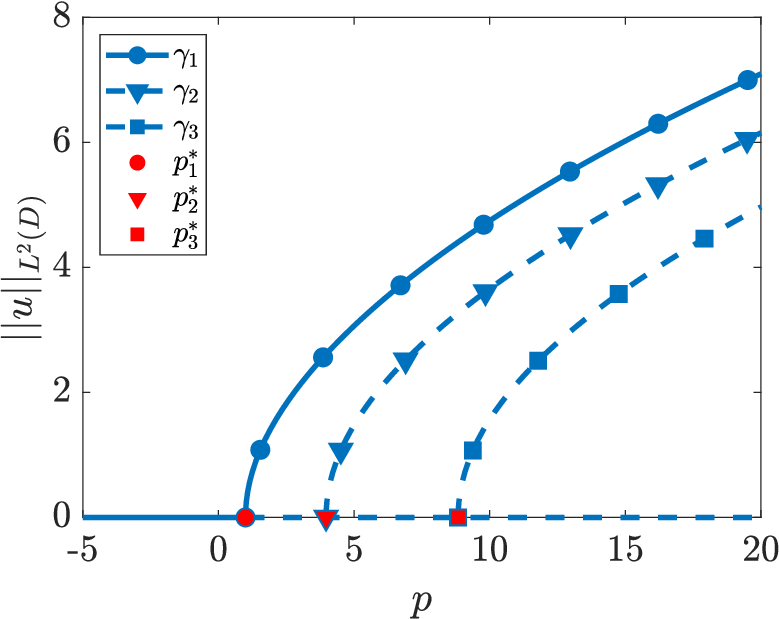}
		\caption{Bifurcation diagram of the Allen--Cahn equation \eqref{eq:AC} on $D=[0,\pi]$: $(p_i^*,0)$, $i=1,2,3$ are the first three bifurcation points and $\gamma_i$ the corresponding nontrivial branches.}
		\label{fig:AC_bif_diag}
	\end{figure}

\section{The Allen--Cahn equation with random coefficients}\label{sect:AC_random}

    Let $\left( \Omega, \mcA,  \Pprob \right)$ be a probability space, where $\Omega$ is the set of elementary events, $\mcA$ a $\sigma$-algebra, and $\Pprob\colon \mcA \rightarrow [0,1]$ a probability measure. Let $D \subset \Rset^d$, $d=1,2,3$ be an open bounded domain with smooth boundary $\partial D$. We consider the following semilinear PDE with random coefficients, 
    \begin{equation}\label{eq:AC_random_general}
		\begin{aligned}
			\partial_t u & = \Delta u + q(\xx,\omega) u - u^3, \quad \quad && \xx \in D, \\
			u & = 0, \quad \quad  && \xx \in \partial D,  \\
            u \vert_{t=0} & = u_0,  \quad &&  \xx \in \overline D, \,\mathbb{P}\text{-almost surely},
		\end{aligned}
	\end{equation}
    where $q: \overline D \times \Omega \rightarrow \Rset$ is a random field (see e.g.\ \cite{lord2014}).  
    \begin{hp}\label{hp:rf} 
        We assume the random field $q(\xx,\omega)$ to be of the following form
        \begin{equation}\label{eq:rf_model}
        q(\xx,\omega) = p + g(\xx,\omega),   \quad p \in \Rset,
        \end{equation}
        where $g: \overline D \times \Omega \rightarrow \Rset $ and the following holds 
        \begin{enumerate} [ref=\ref{hp:rf} (\arabic*)]
        \item \label{assump-A}%
            $g$ is uniformly bounded, i.e.,
            \[
            \exists \ \bar g \in \Rset \, \text{ such that } \, \mathbb{P}( \omega \in \Omega: \lvert g(\xx,\omega) \rvert \leq \bar g \, \, \forall \xx \in \overline{D}) = 1;
            \]
        \item \label{assump-B}%
            $g$ has zero mean, i.e., $\mathbb{E}[g(\xx,\omega)] = 0$, $\forall \xx \in \overline D$.
        \end{enumerate}
    \end{hp} \noindent

    \begin{prop}
        Under Assumption \ref{assump-A}, the problem \eqref{eq:AC_random_general} defines $\mathbb{P}$-almost surely pathwise a dynamical system 
        $\{\mathcal{T},\mathcal{X},\varphi^t_{\omega}\}$ with time interval $\mathcal{T}=[0,+\infty)$, state space $\mathcal{X}=L^2(D;\Rset)$ %H^1_0(D;\Rset)$ 
        and family of evolution operators 
       	\[
    	% \varphi^t: H^1_0(D;\Rset) \rightarrow H^1_0(D;\Rset), \quad \varphi^t(u_0) = u(t,\xx). 
        \varphi^t_{\omega}: L^2(D;\Rset) \rightarrow L^2(D;\Rset), \quad \varphi^t_{\omega}(u_0) = u(t,\xx,\omega) \quad \text{for } \omega \in  \Omega \text{ fixed}. 
    	\]
    \end{prop}
    \begin{proof}
    This follows from \cite[Thm 8.4]{robinson2001} that states global (in time) existence and uniqueness of a solution and its continuous dependence on the initial condition for systems of the type of \eqref{eq:AC_random_general} provided that the polynomial nonlinearity $f(r) = (p+g) r - r^3$ is a $C^1$ function such that 
    \begin{equation}\label{eq:cond_well_posed_1}
    \exists k, \alpha_1,\alpha_2 >0: \quad -k -\alpha_1 r^4 \leq f(r) r \leq k -\alpha_2 r^4  \quad \text{for all } r \in \Rset
    \end{equation}
    and 
    \begin{equation}\label{eq:cond_well_posed_2}
    f'(r) \leq l \quad \text{for all } r \in \Rset.
    \end{equation}
    It clearly holds true that $ f\in C^1$ and \eqref{eq:cond_well_posed_2} follows from the boundedness of $g$ in Assumption \ref{assump-A}. Moreover, applying Young's inequality we obtain  
    \[
    \lvert (p+g)  r^2 \rvert \leq \frac{1}{2}(\lvert p\rvert+\bar g) \Big( \lambda + \frac{r^4}{\lambda} \Big), \quad \lambda >0
    \]
    and it follows that 
    \begin{align*}
    f(r)r & \geq -\frac{1}{2}(\lvert p\rvert+\bar g) \lambda - \frac{1}{2} \left( \frac{\lvert p\rvert+\bar g}{\lambda} + 1 \right) r^4, \\
    f(r)r & \leq \frac{1}{2}(\lvert p\rvert+\bar g) \lambda + \frac{1}{2} \left(  \frac{\lvert p\rvert+\bar g}{\lambda}  -1 \right) r^4.    
    \end{align*}
    Hence, \eqref{eq:cond_well_posed_1} holds true, with $k=\frac{1}{2}(\lvert p\rvert+\bar g) \lambda$, $\alpha_1 = \frac{1}{2} \left(\frac{\lvert p\rvert+\bar g}{\lambda} +1 \right)$, and $\alpha_2 = -  \frac{1}{2} \left(\frac{\lvert p\rvert+\bar g}{\lambda} - 1 \right)$. Note that $k>0$ and $\alpha_1 <0$ for any $\lambda>0$, and $\alpha_2<0$ provided $\lambda>0$ suitably chosen.
    \end{proof}
    From \eqref{eq:rf_model} and Assumption \ref{assump-B} it follows that the random coefficient $q$ in \eqref{eq:AC_random_general} is decomposed into the sum of a zero-mean random field $g$ and a deterministic parameter $p$ that we consider to be unknown. We thus expect that \eqref{eq:AC_random_general} exhibits bifurcations associated with variations of $p \in \Rset$ as in the deterministic case \eqref{eq:AC}. This aspect is investigated in full detail in Section \ref{sect:bif_analysis_het_rand}. 
    Note that the unknown parameter $p$ is actually the expected value of the original random field $q$. Hence, in this setting  bifurcations are triggered by the expected value of the random coefficient $q$. In general, there might be other unknowns in a random field model (usually called hyper-parameters in the UQ community) leading to bifurcations. This is an interesting perspective on possible extensions of our work to other random field models.
   
    Furthermore, we assume the following.
    \begin{hp}\label{hp:finite_dim_noise}
        We assume that the random field $g$ in \eqref{eq:rf_model} is parameterized by a finite number $N$ of independent 
        random variables $Y_n(\omega)$, $n=1,\ldots,N$. We define the random vector $\YY(\omega)= [Y_1(\omega), \ldots, Y_N(\omega)]$  such that 
        \[
        \YY: \Omega \rightarrow \Gamma, \quad \omega \mapsto \yy = [y_1, \ldots, y_N], 
        \]
        where $\Gamma = \Gamma_1 \times \ldots \times \Gamma_N \subset \Rset^N$.
    \end{hp}\noindent
    Due to the independence of the $Y_n$, the probability density function (pdf) associated to the random vector $\YY$ is % given by 
	\[
	\rho: \Gamma \rightarrow [0,+\infty), \quad \rho(\yy) = \prod_{n=1}^{N} \rho_n(y_n),
	\]
    where $\rho_n$ is the pdf of $Y_n$. 
	We then replace the probability space $\left(\Omega, \mcA,  \Pprob \right)$ by $\left(\Gamma, \mathcal{B}(\Gamma), \mathbb{P}_{\YY}\right)$, where $\mathcal{B}(\Gamma)$ denotes the Borel $\sigma$-algebra on $\Gamma$ and $\mathbb{P}_{\YY}$ the distribution of $\YY$ with
    \[
    \mathbb{P}_{\YY}(A) = \int_A \rho(\yy)\mathrm{d}\yy, \quad A \in \mathcal{B}(\Gamma).
    \]  
   Crucially, the infinite-dimensional stochastic problem \eqref{eq:AC_random_general} is now posed on a finite-dimensional probability space $\left(\Gamma, \mathcal{B}(\Gamma), \mathbb{P}_{\YY}\right)$, with set of elementary events $\Gamma \subset \mathbb{R}^N$. With a slight abuse of notation from now on we understand the random field $g$ in \eqref{eq:rf_model} and the PDE solution $u$ as mappings $g \colon \overline D \times \Gamma \rightarrow \mathbb{R}$ and $u\colon \overline D \times \Gamma \rightarrow \mathbb{R}$, respectively. 
    %\textcolor{red}{Crucially, this assumption turns the stochastic problem \eqref{eq:AC_random_general} into a deterministic one with $N$ parameters $\yy \in \Gamma$.} 
    Taking into account this (and the general model assumption in Assumption \ref{hp:rf}), we obtain the following parametric boundary value problem
    \begin{equation}\label{eq:AC_random}
	   \begin{aligned}
		\Delta u + pu + g(\xx,\yy) u - u^3 & = 0, &&\quad \quad   \xx \in D, \\
		u & = 0,  &&\quad \quad  \xx \in \partial D, \, \mathbb{P}_\YY-\text{a.s. }\yy \in \Gamma.
		\end{aligned}
	\end{equation}
	% and, for convenience from now on we consider the solution $u$ as a function $u: \Gamma \rightarrow \Rset$. 
    % \Cadd{where with a slight abuse of notation we use the same letter to denote the random field in \eqref{eq:rf_model} defined over the abstract probability space and the one defined over the image space $\Gamma$. }
    Clearly, the equation resembles the stationary problem associated to the classic Allen--Cahn equation with deterministic coefficients \eqref{eq:AC_stationary} with an additional linear term involving $g$. However, this term does not substantially alter the bifurcation diagram as we will discuss in detail in the following section. Indeed, the trivial equilibrium $u \equiv 0$ is still present and we can follow the steps presented in Section \ref{sect:AC} to study the bifurcations from such trivial equilibrium.

    % We also introduce the Hilbert space $L^2_{\rho}(\Gamma;\Rset)$ of square-integrable functions on $\Gamma$ with respect to the density $\rho(\yy)$, equipped with the associated inner product $\langle \cdot,\cdot\rangle _{L^2_{\rho}}$ and norm $\lVert \cdot \rVert_{L^2_{\rho}}$. 
    
\subsection{Bifurcation analysis}\label{sect:bif_analysis_het_rand}
    Let us introduce the following operator for $\yy \in \Gamma$ fixed, 
    \begin{equation}\label{eq:F_het_rand}
    \widehat F_{\yy} \colon \Rset \times L^2(D;\Rset) \rightarrow L^2(D;\mathbb{R}), \quad \widehat F_{\yy}(p,u)= \Delta u + pu + g(\xx,\yy) u - u^3, %, \quad \text{for } \yy \in \Gamma \text{ fixed}
    \end{equation}
    such that \eqref{eq:AC_random} can be written as $\widehat F_{\yy}(p,u) = 0$ $\mathbb{P}_\YY-\text{a.s. }\yy \in \Gamma$. As mentioned above, the trivial equilibrium occurs also in this case. 
    To study the bifurcations from such equilibrium, we introduce the Frech\'et differential operator of $\widehat F_{\yy}$ evaluated at the trivial equilibrium, that is,  
    \begin{equation} \label{eq:derF_het_rand}
    (\der_u \widehat F_{\yy}) (p,0) \colon   L^2(D;\Rset) \rightarrow L^2(D;\Rset), \quad [(\der_u \widehat F_{\yy}) (p,0)] v = \left(\Delta + p + g(\xx,\yy) \right) v.
    \end{equation}
    The operator $\widehat F_{\yy}$ is again an elliptic operator and we consider it with domain $H^2(D;\Rset) \cap H^1_0(D;\Rset)$, see \cite{kielhofer2014}. It has a discrete spectrum with eigenvalues of finite multiplicity with the rightmost one being simple \cite{henrot2006:eigenvalues}. Then, we assume the following.
    \begin{hp} \label{hp:rf_y} 
    We assume that the coefficient function $g \colon \overline D \times \Gamma \rightarrow \mathbb{R}$ is analytic in its second argument $\yy \in \Gamma$, where $\Gamma \subset \mathbb{R}^N$ is bounded.  
    \end{hp}
    The following holds for the bifurcation points.  
    \begin{prop} \label{prop:bif_rv_het_rand}
        Under Assumption \ref{hp:rf}, \ref{hp:finite_dim_noise}, and \ref{hp:rf_y}, the bifurcation points on the trivial equilibrium of \eqref{eq:AC_random} form a sequence 
        \[
        (p^*_i(\yy),0) \in \Rset \times L^2(D;\Rset) \quad \text{with} \quad p^*_i(\yy) = - \lambda^{g}_i(\yy), \quad i \in \Nset, \, \yy \in \Gamma,
        \] 
        where $\lambda^{g}_i$ is the $i$th-eigenvalue of the operator $\Delta + g(\xx,\yy) \cdot \text{id}$. 
        Moreover, the bifurcation values $p^*_i(\yy)$ arising in correspondence of a simple eigenvalue are realizations of an associated random variable $p_i^*(\YY)$. 
    \end{prop}
    \begin{proof} 
        {In analogy with the classic case \eqref{eq:AC_stationary}, the eigenvalues of $(\der_u \widehat F_{\yy}) (p,0)$ are $p$-shifts of the eigenvalues $\lambda^{g}_i$ and hence $p^*_i(\yy) = -\lambda^{g}_i(\yy)$. Clearly, the eigenvalues depend on the parameter $\yy$ and so do the bifurcation points. Moreover, in \cite[Thm.~3.2]{Chernov:2024} it is shown that simple eigenvalues of this type of operators are analytic in $\yy$, provided Assumption \ref{assump-A},\ref{hp:finite_dim_noise}, and \ref{hp:rf_y} hold}. From this we conclude that they are $\mathcal{B}(\Gamma)$-measurable and thus the corresponding bifurcation values are realizations of suitably defined random variables. 
    \end{proof}

    The application of the Crandall--Rabinowitz theorem recalled in Section \ref{sect:AC} is then straightforward, ensuring existence of local nontrivial branches of equilibria bifurcating from the trivial one in correspondence of simple eigenvalues of the operator \eqref{eq:derF_het_rand} (remember that at least the rightmost eigenvalue is always simple).  
    Further, note that since the eigenfunctions are $\yy$-dependent, the bifurcation direction (cf.\ \eqref{eq:tang_vect_branch}) of each realization of a bifurcation curve is in general different for different values of $\yy$. However, the bifurcations are still of supercritical pitchfork type as in the classic case discussed in Section \ref{sect:AC}. Indeed, the random coefficient appears only in the linear part and hence does not affect the cubic term, the sign of which determines the type of the bifurcation. Finally, the local branches can be globally extended; this follows by applying the same argument as in Section \ref{eq:AC} for $\mathbb{P}_\YY-\text{a.e.\ } \yy \in \Gamma$. 
    
    Moving forward from Proposition \ref{prop:bif_rv_het_rand}, we can write a branch arising at a generic bifurcation point $p^*_i(\yy)$ for $\yy \in \Gamma$ fixed (i.e., corresponding to the generic simple eigenvalue $\lambda_i^{g}(\yy)$) as follows (cf.\ \eqref{eq:symmetric_branch}),    \begin{equation}\label{eq:branch_het_rand_general}
        \gamma^+_i(\yy) = \{( r(s,\yy),u(s,\yy)) \, \vert \, s \in [0,s_{c,\yy}^+], (r(0,\yy),u(0,\yy)) = (p^*_i(\yy) ,0)\},
    \end{equation}
    where we denote the right endpoint of the $s$-interval by $s_{c,\yy}^+$ to highlight the dependence of curve parametrization on the realizations of $\YY$. 
    The analogous holds for $\gamma_i^-(\yy)$. Note that from now on we drop the positive/negative superscript to distinguish between the two symmetric branches (cf.\ \eqref{eq:symmetric_branch}) and consider the positive one only (the discussion can be straightforwardly adjusted to the negative one). This returns a characterization of the curve $\gamma_i(\yy)$ in terms of points $(r(\yy),u(\yy))$, hence moving away from the classic point of view, where the first component of a bifurcation curve is the deterministic bifurcation parameter. 
    %\Cquest{Crucially, this classic perspective is retrieved when performing a shift of magnitude $p^*_i(\yy)$ in the first coordinate of the curve $\gamma_i$ such that each branch can be written for $\yy \in \Gamma$ fixed as  
    %\begin{align}\label{eq:branches_rand}
    	%\gamma_i(\yy) & = \{ \left( (p^*_i(\yy),0 \right)  + \left( p(s),  u(s,\yy) \right) \, \lvert \, s \in [0,s_{c,\yy}], \ (p(0),u(0)) = (0,0) \}.
    %\end{align}  
    %Indeed, in this case each curve is determined by the connected set of points $(p,u(\yy))$ that have only the second component $\yy$-dependent. The first component is then the usual deterministic bifurcation parameter to be explored in the interval $[0,s_{c,\yy}]$. }

    The regularity of the bifurcation branches with respect to $\yy$ is a non-trivial matter as it requires the investigation of the regularity of the solution of a nonlinear problem. 
    We are aware of only a few and very recent works on similar problems \cite{bahn2023semilinear,Chernov:2024b}. However, in our case we also have to deal with the non-uniqueness of the solution which likely requires some additional work.
    We plan to address this point in future studies. 
    For this work, especially to allow for the use of the PC method (see Section \ref{sect:PC}), we assume that each branch $\gamma_i(\yy)$ can be interpreted as a realization of a well-defined and sufficiently regular stochastic process indexed by the curve parameter $s$. Specifically, let us then introduce the weighted Hilbert space $L^2_{\rho}(\Gamma;\Rset)$ of square-integrable real-valued functions on $\Gamma$ with respect to $\Pprob_{\YY}$, equipped with the associated inner product $\langle \cdot,\cdot\rangle _{L^2_{\rho}}$ and norm $\lVert \cdot \rVert_{L^2_{\rho}}$. We assume the following. 
    \begin{hp}\label{hp:u_regularity}
    We assume that each random branch $\gamma_i(\YY)$ is a $\mathbb{P}_{\YY}$-square integrable stochastic process $\left\{ {\bf X}_i(s)\right\}_{s \in [0,S]}$, i.e.~$r(s,\cdot) \in L^2_{\rho}(\Gamma;\Rset)$ and $u(s,\cdot) \in L^2_{\rho}(\Gamma;L^2(D;\Rset))$ for all $s \in [0,S]$, with realizations $\left(r(s,\yy),u(s,\yy) \right) \in \gamma_i(\yy)$. 
    %\begin{equation}\label{eq:curve_as_stoch_proc}
    %    \left\{ {\bf X}_i(s)\right\}_{s \in [0,S]} = 
    %    \left\{\left(r(s,\YY),u(s,\YY) \right) \right\}_{s \in [0,S]},
    %\end{equation}
    %where 
    We denote by $S>0$ the right endpoint of the curve parametrization interval.
    \end{hp}
    \noindent
    Note that $S$ has to be chosen such that all the realizations of the curve are well-defined, i.e., such that no bifurcation point appears on the realizations of the branch within the parametrization interval $[0,S]$ (cf.\ the discussion after \eqref{eq:symmetric_branch}). In particular, since $s=0$ corresponds to the bifurcation point, the $\mathbb{P}_{\YY}$-square integrability is assumed for the bifurcation point as well.
    %Furthermore, we observe that each branch can be interpreted as realization of a random parametrized curve $\{ {\bf X}(s) \}_{s \in [0,+\infty)}$ indexed by the arclength parameter $s$ such that
    %\begin{equation}\label{eq:curve_as_stoch_proc}
    %    \left\{ {\bf X}(s)\right\}_{s \in [0,+\infty)} = 
    %    \left\{\left(r(s,\YY),u(s,\YY) \right) \right\}_{s \in [0,+\infty)}.
    %\end{equation}
    % Indeed, for a fixed $s \in [0,+\infty)$, the couple $\left(r(s,\YY),u(s,\YY) \right) $ is a random variable as it is solution of $\widehat F(r,u(r);\yy) = 0$ for $r(\yy) \geq p^*_i(\yy)$. 
    Then, since the expected value of a stochastic process is a well-defined object, we introduce the concept of  \textit{mean bifurcation curve}. 
    \begin{defn}[Mean bifurcation curve]
    The mean bifurcation curve is 
    \begin{equation}\label{eq:mean_bif_curve}
        \overline{\gamma}_i \colon = \left\{ \left(\bar r (s), \bar u(s) \right)  = \left(\mathbb{E}[r(s,\YY)], \mathbb{E}[u(s,\YY) ] \right), \, s \in [0,S], \left(\bar r(0), \bar u(0)\right) = (\mathbb{E}[p^*_i(\YY)],0) \right\}, 
    \end{equation}
    i.e., the curve that bifurcates at the expected value of the bifurcation point, with the two components being the expected value of the stochastic process 
    $\left\{ {\bf X}_i (s)\right\}_{s \in [0,S]}$ introduced in Assumption \ref{hp:u_regularity}. 
    \end{defn}    
    Similarly, we can consider higher order moments of the stochastic process and define corresponding ``characteristic'' curves to describe each random bifurcation branch $\gamma_i(\YY)$. 
    
    To summarize, the Allen--Cahn equation with random coefficients \eqref{eq:AC_random_general} together with Assumption \ref{hp:rf}, \ref{hp:finite_dim_noise}, and \ref{hp:rf_y} exhibits supercritical pitchfork bifurcations: the corresponding bifurcation points are random variables and bifurcation curves are random curves. The UQ analysis of these objects is discussed in Section \ref{sect:AC_het_rand}, where we introduce our proposed numerical methodology based on the PC expansion. 
    However, the UQ analysis turns out to be possible at an analytic level in the particular case of spatially-homogeneous random coefficients, i.e.\ $q(\xx,\yy) = q(\yy)$. 
    We detail this discussion in the following section. 

    \section{Spatially-homogeneous random coefficients}\label{sect:AC_hom_rand}
	
    In this section we focus on the case of spatially-homogeneous randomness: We set $g(\xx,\yy) = g(\yy)$ in \eqref{eq:AC_random}. 
    In this case we can refine the results on bifurcation points (cf.\ Proposition \ref{prop:bif_rv_het_rand}) and bifurcation curves. 
    In particular, we obtain a closed form expression for the pdf of the bifurcation points, and show that the nontrivial bifurcation branches are shifts of a common reference branch.
    
    Let us first observe that, for $\yy \in \Gamma$ fixed, the operator $(\der_u \widehat F_{\yy}) (p,0)$ in \eqref{eq:derF_het_rand} is simply a shifted Laplacian. 
    Then, for the bifurcation points the following holds.
    \begin{prop} \label{prop:bif_rv_hom_rand}
        Under Assumption \ref{hp:rf} \ref{hp:finite_dim_noise}, and \ref{hp:rf_y}, the bifurcation points on the trivial equilibrium of \eqref{eq:AC_random} with $g(\xx,\yy) = g(\yy)$ form a sequence 
        \begin{equation}\label{eq:rv_bif_pt}
        (p^*_i(\yy),0) \in \Rset \times L^2(D;\Rset) \quad \text{with} \quad p^*_i(\yy) = -\lambda_i - g(\yy), \quad i \in \Nset,
        \end{equation}
        where $\lambda_i$ is the $i$th-eigenvalue of the Laplacian with homogeneous Dirichlet boundary conditions. 
        Moreover, the bifurcation values $p^*_i(\YY)$ are random variables 
        %with finite second order moment (i.e.\ $p_i^*(\YY) \in L^2_{\rho_g}(\Gamma;\Rset)$) 
        with associated pdf
        \begin{equation}\label{eq:pdf_bif_pt}
	       \rho_{p_i^*}(\yy) = \rho_g (- \lambda_i - \yy),
	    \end{equation}        
        where $\rho_g$ is the pdf of the random variable $g(\YY)$.
    \end{prop}
    \begin{proof}
    Let $\yy \in \Gamma$ be fixed. The eigenvalues of $(\der_u \widehat F_{\yy}) (p,0)$ are shifts of the eigenvalues of the Laplacian on $D$ subject to homogeneous Dirichlet boundary conditions. 
    Hence \eqref{eq:rv_bif_pt} follows straightforwardly. Moreover, we observe that bifurcation values are obtained by shifting the random variable $g(\YY)$ by the deterministic quantity $\lambda_i$. 
    From this we conclude that bifurcation points are random variables and that their pdf is the shifted pdf of $g(\YY)$, i.e.\ \eqref{eq:pdf_bif_pt}.  
    \end{proof}
    
    Regarding the bifurcation branches we observe the following: Any nontrivial branch of equilibria \eqref{eq:branch_het_rand_general} is a random shift of a deterministic reference bifurcation branch. %;  we refer to it as reference branch in the following. 
    A natural choice for the reference branch $\gamma_i^{\text{ref}}$ is the curve corresponding to the case $g(\yy) \equiv 0$, which starts at the point $(p^{*,\text{ref}}_i,0) = (-\lambda_i,0)$
    \begin{equation}\label{eq:branch_ref}
        \gamma^{\text{ref}}_i = \left\{ \left(p(s),u(s) \right) \, \vert \, s \in [0,S], (p(0),u(0))= (-\lambda_i,0) \right\}.
    \end{equation}
    Then, the following holds.
    
    \begin{prop}\label{prop:branch_hom_rand} 
    Let Assumption \ref{hp:rf} and \ref{hp:finite_dim_noise} hold and consider \eqref{eq:AC_random} with $g(\xx,\yy) = g(\yy)$. Then, the nontrivial branch of equilibria arising at the bifurcation point $(p^*_i(\yy),0) = (-\lambda_i - g(\yy), 0)$ 
    is given by 
    \begin{equation}\label{eq:branch_hom_rand_shift}
    \gamma_i(\yy) = \left\{ \left( -g(\yy) , 0 \right) + \left(p(s),u(s) \right) \, \vert \, \left(p(s),u(s) \right) \in \gamma^{\text{ref}}_i \right\},
    \end{equation}
    with $\gamma^{\text{ref}}_i$ as in \eqref{eq:branch_ref}. 
    Furthermore, the mean bifurcation curve \eqref{eq:mean_bif_curve} coincides with the reference branch. 
    \end{prop}
    \begin{proof}
    Let $\yy \in \Gamma$ be fixed. The branch arising at the bifurcation point $(p^*_i(\yy),0) = (-\lambda_i - g(\yy), 0)$  is (cf.\ \eqref{eq:branch_het_rand_general})
    \[
        \gamma_i(\yy) = \{( r(s,\yy),u(s,\yy)) \, \vert \, s \in [0,S], (r(0,\yy),u(0,\yy)) = (-\lambda_i - g(\yy),0)\}.
    \]
    By definition, the points $(r(\yy),u(\yy)) \in \gamma_i(\yy)$ are solutions of $\widehat F_{\yy}(r,u) = 0$ for $r(\yy) \geq p^*_i(\yy)$. The solution points are translation-invariant: The points $(p, u(p)) = (r(\yy)+g(\yy),u(r(\yy)+g(\yy)))$ solve the same equation for $p = r+g(\yy) \geq -\lambda_i$. Crucially, such points form the reference branch \eqref{eq:branch_ref}.
    The statement on the mean bifurcation curve follows from the fact the random field $g$ has zero expected value.
    \end{proof}

    The above two propositions are crucial for performing the UQ analysis of bifurcations. This is the focus of the following section. 

    \subsection{Uncertainty quantification analysis of bifurcations} \label{sect:UQ_hom_rand}
    In this section we detail how to assess the impact of the randomness in the coefficients on the overall behavior of the dynamical system associated to \eqref{eq:AC_random}. To ease the presentation we introduce the following example. 
    
    Let us consider the Allen--Cahn equation \eqref{eq:AC_random} with $g(\yy) = y$, $y \in \Gamma \subset \Rset$, where $N=1$, i.e.\ we assume that the randomness is modeled by a single random variable. The simplification is made to keep the presentation simple yet informative enough. We set $D = [a,b] \in \Rset$, $a,b \in \Rset$, $a<b$ and perform a uniform discretization in space with $m+2$ grid points $x_j, j=0,\ldots,m+1$ with $x_0 = a$ and $x_{m+1}=b$ with grid size $h = (b-a)/(m+1)$. 
    We denote by $u_j(y)$ the approximation of the solution at the grid points, i.e.\ $u_j(y) \approx u(x_j,y)$, $y \in \Gamma$ and introduce the vector $\uu(y) = [u_0(y), \ldots, u_{m+1}(y)]^{\text{T}} \in \Rset^{m+2}$, with $u_0(y)=u_{m+1}(y)=0$. 
	Then, we consider the following discretized form of \eqref{eq:AC_random} in the interior of $D$
	\begin{equation}\label{eq:AC_discrete_hom}
		K \uu + \left(p  + y \right)   \uu - \uu^3 = \mathbf{0}, \quad p \in \Rset, \, y \in \Gamma
	\end{equation}
	where $K\in\Rset^{m \times m}$  is the standard central finite-difference discretization matrix of the Laplacian with homogeneous Dirichlet boundary conditions (note that any other discretization scheme is possible). The results reported in this section are obtained with $D=[0,\pi]$ and $m=20$.

    \paragraph{Bifurcation points}
    We consider two input random variables that satisfy Assumption \ref{hp:rf_y}: a uniform random variable on $[-1,1]$ and a standard Gaussian truncated on the interval $[-2,2]$. The distribution of the corresponding first three bifurcation values is known analytically and follows from Proposition \ref{prop:bif_rv_hom_rand}. In Figure \ref{fig:eig_hom_rand} panel (a) and panel (b) we show the pdfs corresponding to the uniform and the Gaussian case, respectively. We display with dashed lines the corresponding eigenvalue with opposite sign $-\lambda^K_i, i=1,2,3$ of the discretization matrix $K$ of the Laplacian. Note that in this case the eigenvalues are known analytically and employing the spatially discretized version of the PDE is not required.

    \begin{figure}[tb]
		\begin{minipage}{0.5\textwidth}
			\centering
			\includegraphics[scale=0.49]{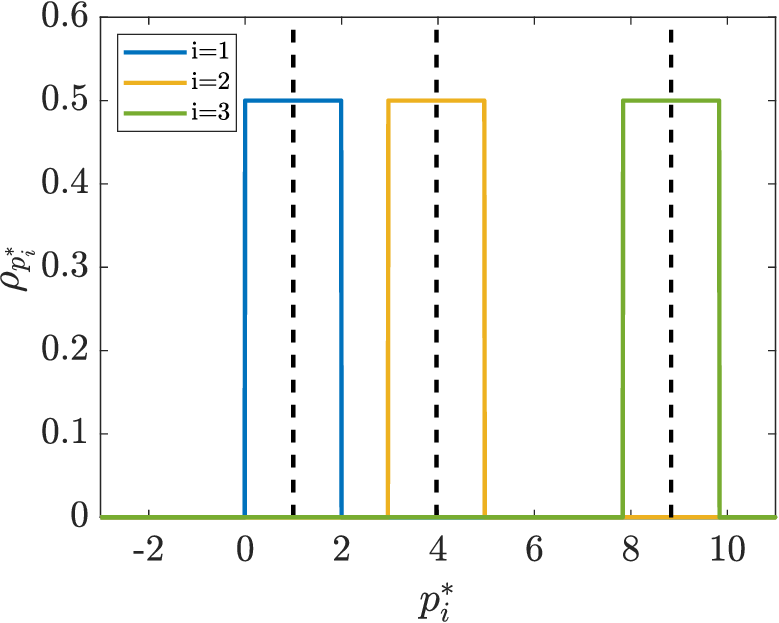}
			\subcaption{$Y \sim \mcU([-1,1])$.}
		\end{minipage}% \hspace{-4em}
		\begin{minipage}{0.5\textwidth}
			\centering
			\includegraphics[scale=0.49]{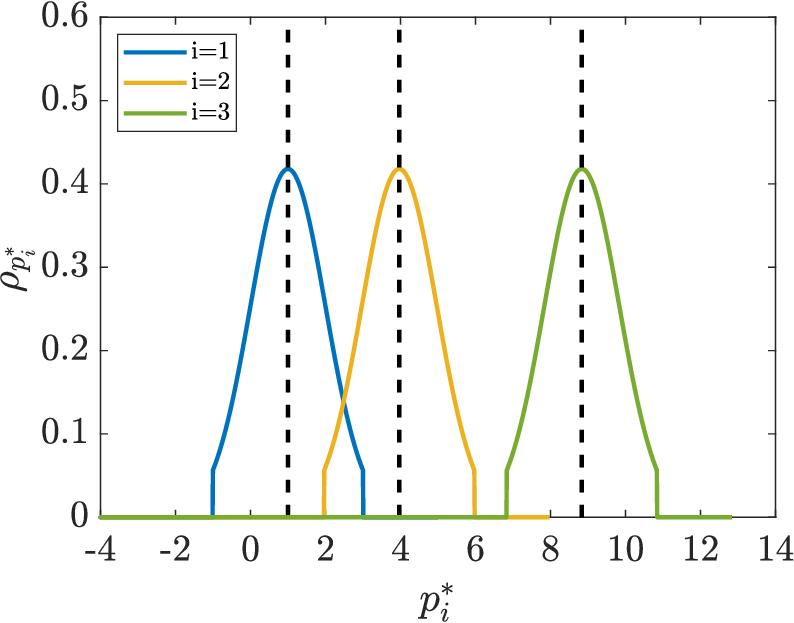}
			\subcaption{$Y \sim \mcN(0,1)$ truncated on $[-2,2]$.}
		\end{minipage}
		\caption{Allen--Cahn equation \eqref{eq:AC_random} with $g(x,y)=y$. Pdf of the first three bifurcation value $p_i^*(Y)$, $i=1,2,3$; the dashed lines mark the corresponding eigenvalues of the matrix $-K$, cf.\ \eqref{eq:pdf_bif_pt}.}
		\label{fig:eig_hom_rand}
	\end{figure}
 
    %Note that the curve overlap for subsequent bifurcation points in panel (b) depends on the range of variability of the input random variable. However, the possible overlap does not affect the probabilistic assessment of the stability region of the trivial equilibrium for which one has to consider the cdf of the first bifurcation point only. 
	
    \paragraph{Probability of bifurcating}% \paragraph{Stability region of the trivial equilibrium}
    Given a value $\bar{p} \in \Rset$, the probability of the first bifurcation point to happen at values $p \leq \bar{p}$ is 
	\begin{equation}\label{eq:cdf}
	\Psi_{p_1^*}(\yy) := \Pprob_\YY(p_1^*(\yy)\leq\bar{p}) = \int_{-\infty}^{\bar{p}} \rho_{p_1^*} (\yy) \, \mathrm{d} \yy,
	\end{equation}
    which, by definition, is the cumulative density function (cdf) of $p_1^*(\YY)$.
    Moreover, since the boundary of the stability region of the trivial equilibrium is identified by the first bifurcation point $p_1^*$, the quantity in \eqref{eq:cdf} can be also regarded as the probability of the trivial equilibrium to change its stability type at values $p \leq \bar{p}$ or the probability of the trivial equilibrium to be stable at values $p \leq \bar{p}$. Finally, note that for the example considered in this section the distribution of $g(\YY)$ is known analytically and the above probability can be computed analytically. 
    
    % As recalled in Section \ref{sect:AC}, the boundary of the stability region of the trivial equilibrium is identified by the first bifurcation point $p_1^*$, that we have shown to be a random variable in Proposition \ref{prop:bif_rv_hom_rand}. Hence, such boundary can be assessed only probabilistically. In particular, given a value $\bar{p} \in \Rset$, the probability of the first bifurcation point to happen at values $p \leq \bar{p}$ is 
	%\begin{equation}\label{eq:cdf}
	%\Psi_{p_1^*}(\yy) := \Pprob_\YY(p_1^*(\yy)\leq\bar{p}) = \int_{-\infty}^{\bar{p}} \rho_{p_1^*} (\yy) \, \mathrm{d} \yy,
	% \end{equation}
    % which, by definition, is the cumulative density function (cdf) of $p_1^*(\YY)$. Moreover, note that the same quantity can be also regarded as the probability of the trivial equilibrium to change its stability type at values $p \leq \bar{p}$ or the probability of the trivial equilibrium to be stable at values $p \leq \bar{p}$. Note that in this case the distribution of $g(\YY)$ is known analytically and the above probability can be computed exactly. 

    %Further, due to the fact that each nontrivial random branch has shown to be a random translation of the reference branch \eqref{eq:branch_ref}, we can conclude that all realizations of a nontrivial branch inherit the stability type from the corresponding reference branch. Hence, we can borrow the discussion in Section \ref{sect:stability} and conclude that the first nontrivial random branch is stable, whereas the subsequent ones are unstable

    \paragraph{Nontrivial bifurcation curves}
    Sampling the nontrivial bifurcation curves means computing different instances of the random branch $\gamma_i(\YY)$ corresponding the different realizations $\yy \in \Gamma$ of the random variable $\YY$.
    In general, this requires running multiple times a suitable numerical method for the approximation of each curve. The class of methods to this aim are the so-called continuation methods. These methods typically rely on a predictor-corrector strategy to extend the solution of the nonlinear problem at one value of the bifurcation parameter to another parameter value. See Appendix \ref{appendix:num_cont} for a brief recall and e.g.\ \cite{dankowicz2013,kuznetsov2013,uecker2021} for a more extensive treatment.
    In the case of spatially-homogeneous randomness, thanks to Proposition \ref{prop:branch_hom_rand}, we need to run only once a continuation algorithm for the approximation of the reference branch \eqref{eq:branch_ref}. All the other samples can then be obtained by simply shifting the reference curve according to \eqref{eq:branch_hom_rand_shift}. Crucially, the cost of the resulting sampling strategy amounts to the cost of computing \textit{one} deterministic branch. 

    \begin{figure}
		\begin{minipage}[t]{0.48\textwidth}
			\centering
			\includegraphics[scale=0.49]{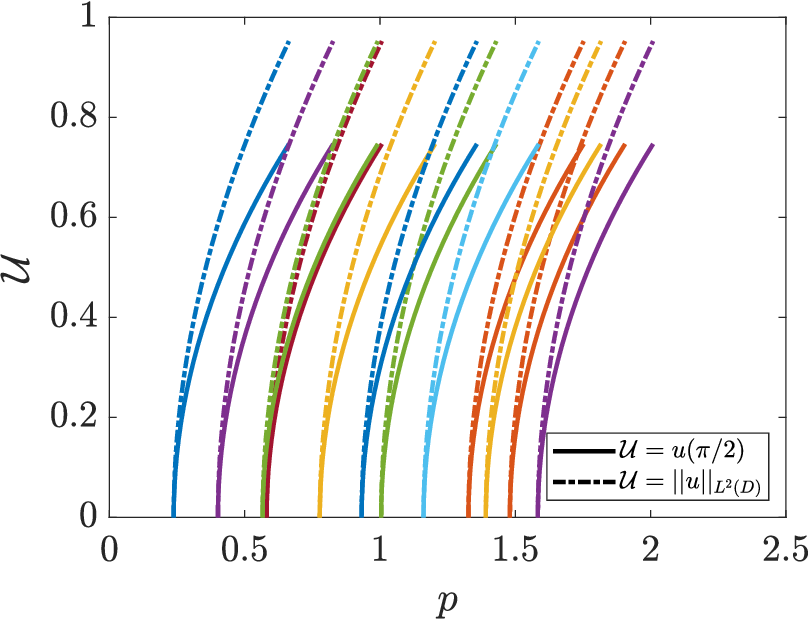}
			\subcaption{Samples of $\gamma_1(Y)$ for two observables $\mcU$ associated to its $u$-component. }
		\end{minipage}\hfill
		\begin{minipage}[t]{0.5\textwidth}
			\centering
			\includegraphics[scale=0.49]{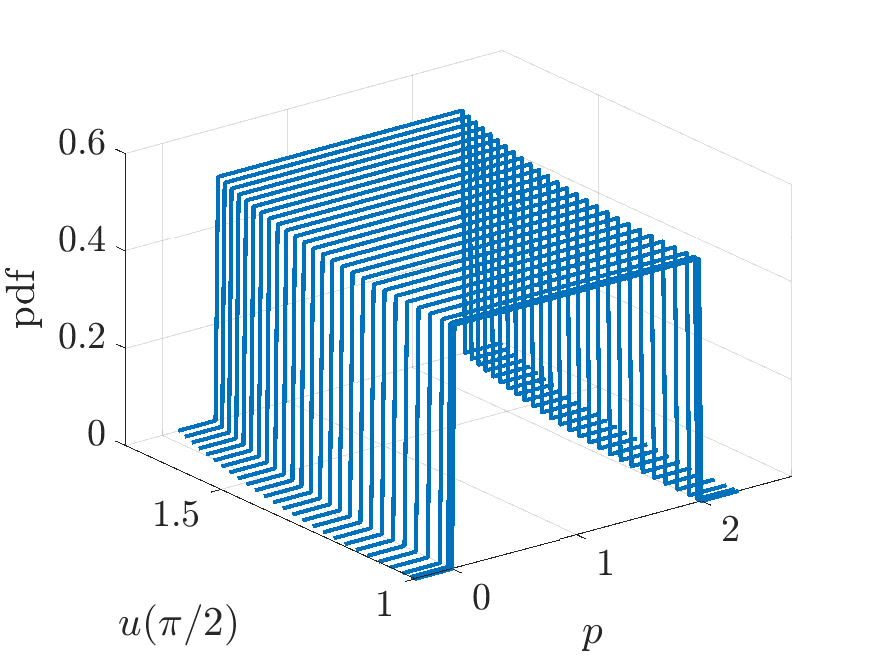}
			\subcaption{Pdf along $\gamma_1(Y)$, cf.\ \eqref{eq:pdf_branch}.}
		\end{minipage}
	\caption{Allen--Cahn equation \eqref{eq:AC_random} with $g(x,y)=y$, $Y \sim \text{Unif}(-1,1)$. First nontrivial bifurcation branch $\gamma_1(Y)$.}
	\label{fig:branch_hom_rand}
	\end{figure}
    
    In Figure \ref{fig:branch_hom_rand} we provide two ways to visualize the first nontrivial branch $\gamma_1(Y)$. Specifically, in panel (a) we consider two scalar observables, the value of the solution at the center of $D$ (solid line) and the $L^2(D)$-norm of the solution (dash-dotted line), and display several realizations of them according to \eqref{eq:branch_hom_rand_shift}. The sample curves are depicted with different colors to highlight the fact that they correspond to different realizations of $Y$. For the approximation of the reference bifurcation branch we use the numerical continuation software Continuation Core and Toolboxes (COCO) \cite{dankowicz2013} (version October 26, 2023) available at \cite{coco:sourceforge}. 
    Moreover, we display the pdf associated to the bifurcation curve, namely the pdf of the following random sets along the branch which correspond to discrete values of the parameter $s$, i.e.
	\begin{equation}\label{eq:pdf_branch}
		\{(r(s_l,y),u(s_l,y)) \in \gamma_1(Y), s_l \in [0,S] \}.
	\end{equation} 
    In panel (b), we plot such pdf for the observable $\mathcal{U}= u(\pi/2)$.
    %\Cquest{In panel (b), we display the pdf associated to the bifurcation curve, namely the pdf of the following random sets along the branch which correspond to discrete values of the parameter $s$, i.e.
	%\begin{equation}\label{eq:pdf_branch}
	%\{(r(s_l,y),u(s_l,y)) \in \gamma_1(Y), s_l \in [0,S] \}.
	%\end{equation}}
	As shown in Proposition \ref{prop:branch_hom_rand}, the $u$-component of a branch is deterministic. Hence the pdf of the random sets is actually a univariate function. Moreover, it follows that the pdf of the bifurcation point is simply ``transported'' along the reference curve $\gamma_1^{\text{ref}}$. 	
	
	Further, in Figure \ref{fig:mean_sol_hom_rand} panel (a) we display the solution of \eqref{eq:AC_random} along the first bifurcation branch, i.e.\ for increasing values of $p\geq p_1^*$. In particular, the trivial solution is obtained for $p=p_1^*$ and the curve with the largest maximum value corresponds to the largest value of $p$ considered (i.e., the point where we stop the numerical continuation method for the computation of the reference branch). Since the $u$-component of a branch is deterministic, this plot looks the same for each realization of $Y$. Finally, in Figure \ref{fig:mean_sol_hom_rand} panel (b) we show the mean bifurcation diagram up to the first bifurcation branch, i.e.\ the bifurcation diagram representing the trivial equilibria and the first mean bifurcation branch. 
	
	\begin{figure}
		\begin{minipage}[t]{0.48\textwidth}
			\centering
			\includegraphics[scale=0.45]{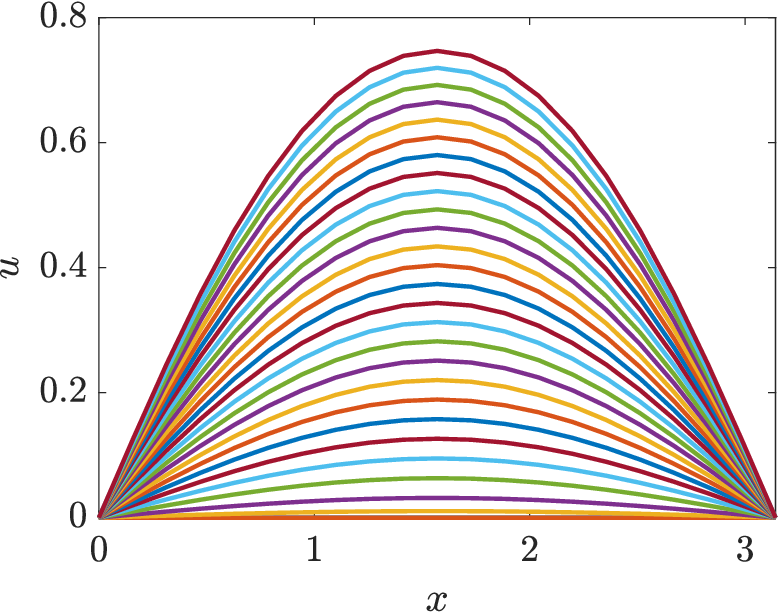}
			\subcaption{Solution along the first bifurcation branch. The different curves correspond to increasing values of the bifurcation parameter $p$.}
		\end{minipage}
		\begin{minipage}[t]{0.48\textwidth}
			\centering
			\includegraphics[scale=0.45]{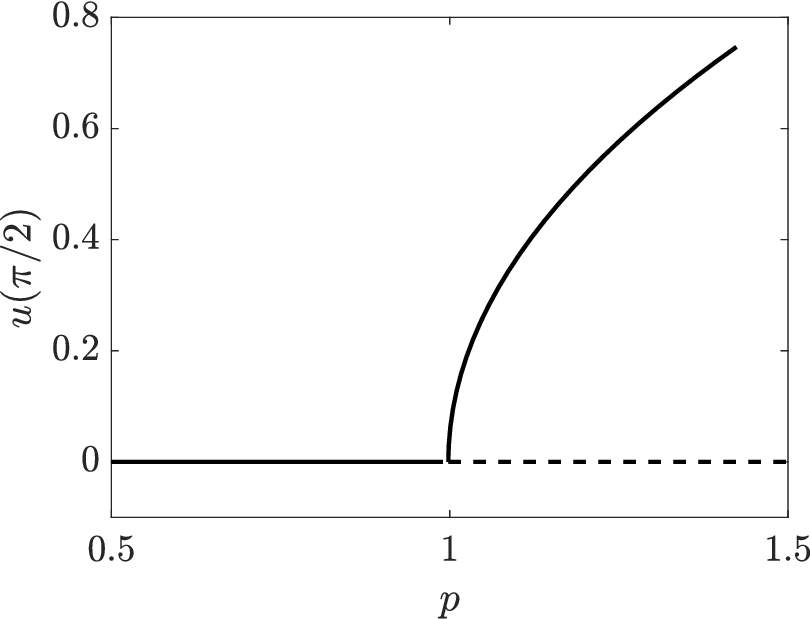}
			\subcaption{Mean bifurcation diagram.}
		\end{minipage}
		% \hspace{-4em}
	\caption{Allen--Cahn equation \eqref{eq:AC_random} with $g(x,y)=y$, $Y \sim \text{Unif}(-1,1)$. }
	\label{fig:mean_sol_hom_rand}
	\end{figure}
% \end{example}

%%%%%%%%%%%%%%%%%%%%%%%%%%%%%%%%%%%%%%%%%%%%%%%%%%%%%%%%%%%%%%%%%%%%%%%%%%%%%%%%%%%%%%%%%%%%%%%%%%%%%%%%%%%%%%%%%%%%%%%%%%%%%%%%%%%%%%%%%%%%%%%%%%%%%%%%%%%%%%%%%%%%%%%%%%%%%%%%%%%%%%%%%%%%%%%%%%%%%%%%%%%%%%%%%%%%%%%%%%%%%%%%%%%%%%%%%%%%%%%%%%%%%%%%%%%%%%%%%%%%%%%%%%%%%%%%%%%%%%%%%%%%%%%%%%%%%%%%%%%%%%%%%%%%%%%%%%%%%%%%%%%%         
\section{Spatially-heterogeneous random coefficients}\label{sect:AC_het_rand}

    In this section we describe the numerical approach that we propose for the UQ analysis of bifurcations of \eqref{eq:AC_random} in the general case of a spatially-heterogeneous random coefficient $g(\xx,\YY)$. 
    
    First, we highlight the fact that in this case, in contrast to the spatially-homogeneous case discussed in the previous section, the distribution of each bifurcation point $p^*_i(\YY)$ is not known analytically (cf.\ Proposition \ref{prop:bif_rv_het_rand}) and needs to be estimated. This requires the approximation of the pdf of the simple eigenvalues of the random operator $\Delta +g(\xx,\YY) \cdot \text{id}$ and requires the solution of a set of eigenvalue problems corresponding to a number of samples of $\YY$.  
    Second, the approximation of the nontrivial bifurcation branches $\gamma_i(\YY)$ is not as trivial as in the spatially-homogeneous case. Computing curve samples would require to solve multiple times a nonlinear equation by continuation methods. Moreover, the approximation of the mean bifurcation curve $\overline \gamma_i $ requires the approximation of an expected value, cf.\ \eqref{eq:mean_bif_curve}.
    A naive approach based on a Monte-Carlo average would require a large number of sample curves and hence running several times a numerical continuation algorithm that is itself quite expensive. 
    However, such an approach would perhaps still be computationally feasible for the case of the Allen--Cahn equation, but certainly shows severe limitations for more complicated problems. % VWe discuss more on this at the end of Section \ref{sect:UQ_het_rand}. 
    Hence, we propose an approach for the UQ analysis of bifurcations based on the gPC expansion. In the following we recall the basis of the gPC method and set up the framework to be used in the subsequent section.
    
	\subsection{The generalized Polynomial Chaos expansion}\label{sect:PC}
    %\Cquest{Let us first introduce the weighted Hilbert space $L^2_{\rho}(\Gamma;\Rset)$ of square-integrable real-valued functions on $\Gamma$ with respect to the density $\rho$, equipped with the associated inner product $\langle \cdot,\cdot\rangle _{L^2_{\rho}}$ and norm $\lVert \cdot \rVert_{L^2_{\rho}}$. }  
    %Then, t
    The gPC expansion of a real-valued function $f$ with finite second-order moment (i.e.\  $f \in L^2_{\rho}(\Gamma;\Rset)$, see Section \ref{sect:bif_analysis_het_rand}) is a representation of $f$ in terms of orthonormal polynomials with respect to the $L^2_\rho$-inner product in \eqref{eq:PC_generic}. 
    %$\rho(\yy) \mathrm{d}\yy$-orthonormal polynomials. % It was first proposed for Gaussian measures in \cite{wiener1938} and extended to the non-Gaussian case in \cite{XiuKarniadakins2002:gPC} with the name of gPC. 
    The truncated expansion is mean-square convergent, i.e.\ convergent with respect to $\lVert \cdot \rVert_{L^2_{\rho}}$; this follows from the Cameron--Martin theorem for the Gaussian case (see \cite{CameronMartin1947}) and was shown in the general case in \cite{ernst2012}. 
    
    We construct a basis of $ L^2_{\rho}(\Gamma;\Rset)$ consisting of 
    %$\rho(\yy) \mathrm{d}\yy$
    $L^2_\rho$-orthonormal polynomials, that can be written as products of univariate orthonormal polynomials, since $ L^2_{\rho}(\Gamma;\Rset) = \bigotimes_{n=1}^N L^2_{\rho_n}(\Gamma_n;\Rset)$, see e.g.\ \cite{soize2004}. 
    In formulas, let $\aalpha = (\alpha_1, \ldots, \alpha_N) \in \Nset^N_{\geq 0}$ be a multi-index and $\psi_{\aalpha} \in L^2_{\rho}(\Gamma;\Rset)$ an element of the basis. Then, it holds 
	\[
	\psi_{\aalpha}(\yy) = \prod_{n=1}^N \psi_{\alpha_n}(y_n) \quad \text{such that} \quad \langle \psi_{\aalpha}, \psi_{\boldsymbol{\beta}} \rangle_{L^2_{\rho}} = \delta_{\aalpha\, \boldsymbol{\beta}},
	\]
	where $\psi_{\alpha_n}$ are orthonormal polynomials of degree less or equal than $\alpha_n$ in the variable $y_n$. 
	Then, the expansion of $f$ in terms of this basis is the so-called gPC expansion and has the following form 
	\begin{equation}\label{eq:PC_generic}
		f(\yy) = \sum_{\boldsymbol{\alpha} \in \Nset^N_{\geq 0}} \widehat{f}_{\boldsymbol{\alpha}} \psi_{\boldsymbol{\alpha}} (\yy)
		\quad \text{with} \quad \widehat{f}_{\aalpha} = \langle f, \,\psi_{\aalpha} \rangle_{L^2_{\rho}} = \int_{\Gamma} f(\yy) \psi_{\aalpha}(\yy) \rho(\yy) \textrm{d} \yy.
	\end{equation}
    According to \cite{XiuKarniadakins2002:gPC}, the families of polynomials to be employed in the representation can be chosen from the Askey scheme. For example, Legendre polynomials are chosen for uniform random variables, Hermite for Gaussians, Laguerre for gamma random variables, and Jacobi polynomials for beta random variables; we refer to \cite{ernst2012} for an extended discussion.
    	
	The rate of convergence of the gPC approximation and hence the truncation of the expansion depends on the regularity of $f$, see e.g.\ \cite{xiu2010:book}. A common choice is to retain in the sum only the polynomials of total degree at most $z \in \Nset$, which is equivalent of selecting a set of multi-indices $ \Lambda \subset \Nset^N$ of the form 
	$\Lambda = \{ \aalpha \in \Nset^N \, \vert \,  \sum_{n=1}^N \alpha_n  \leq z\}$. Other choices for the multi-index set $\Lambda$ are reported in \cite{back2011,Beck2012:optimal.pol.,shen.wang2010}; an adaptive selection of the multi-index set is also possible, see e.g.\ \cite{sudret2011:adaptive.PCE.with.reg}. 
	Then, the corresponding gPC approximation $f_{\text{PC},\Lambda}(\yy)$ of $f$ is defined as follows  
	\begin{equation}\label{eq:PC_truncated}
		f(\yy) \approx f_{\text{PC},\Lambda}(\yy) := \sum_{\boldsymbol{\alpha} \in \Lambda} \widehat{f}_{\boldsymbol{\alpha}} \psi_{\boldsymbol{\alpha}} (\yy),
	\end{equation}
    with $\Lambda$ a suitably chosen multi-index set.

    For the computation of the gPC approximation we use the Sparse Grids Matlab Kit \cite{piazzola2022:sparse.grids,sgmk:github}, a Matlab package providing an implementation of sparse grids for approximating high-dimensional functions and for surrogate-model-based uncertainty quantification, that is co-developed and maintained by one of the authors.  
    In particular, we highlight the fact that the computation of the expansion coefficients $\widehat{f}_{\aalpha}$ is based on the introduction of a sparse-grids-based approximation in terms of Lagrange polynomials \cite{babuska2010,bungartz2004,xiu2005} to be converted to the equivalent gPC expansion. This approach was proposed in \cite{formaggia2013} and we briefly recall it in Appendix \ref{appendix:sg}. Note that the computation of the coefficients can be performed in several ways, e.g., by quadrature \cite{xiu2007:PC.coeff}
    %, least squares fitting \cite{sudret2011:adaptive.PCE.with.reg}, 
    or Galerkin methods \cite{lemaitre:2010}.
    Instead, the methodology implemented in the Sparse Grids Matlab Kit avoids evaluating high-dimensional integrals according to \eqref{eq:PC_truncated} and has been shown to lead to an approximation of the coefficients that are more accurate than if we would employ a sparse-grids-based quadrature directly, see \cite{constantine2012}.
    Furthermore, this approach is fully non-intrusive: it is based on the stochastic collocation paradigm and as such it requires only values of $f$ at some appropriately selected points in $\Gamma$. This is particularly advantageous in our case as it allows to make use of deterministic solvers for eigenvalue problems for the approximation of bifurcation points and numerical continuation methods for the approximation of random %\Cquest{the computation of sample}
    bifurcation curves in a black-box manner. 

    % Moreover, the use of sparse grids mitigates the computational cost of the estimation of the PC coefficients which typically requires evaluation of high-dimensional integrals.
    Finally, note that the first coefficient in \eqref{eq:PC_truncated} gives the expected value of $f$. Indeed, it holds that 
    $
	\mathbb{E}[f(\YY)] = \langle f, \, \psi_{\bf{0}} \rangle _{L^2_{\rho}} = \sum_{\boldsymbol{\alpha}} \widehat{f}_{\boldsymbol{\alpha}} \langle \psi_{\boldsymbol{\alpha}}, \, \psi_{\bf{0}} \rangle _{L^2_{\rho}} =  \widehat{f}_{\bf{0}}
	$
	thanks to the orthogonality of the basis functions. In a similar way, higher order moments can be derived from the expansion coefficients, leading to an extremely efficient way of computing e.g.\ the mean bifurcation curve \eqref{eq:mean_bif_curve}. 

    \subsection{Uncertainty quantification analysis of bifurcations } \label{sect:UQ_het_rand}

    In this section we discuss the forward UQ analysis of bifurcation points and nontrivial bifurcation curves along the lines in Section \ref{sect:UQ_hom_rand} for the case of spatially-homogeneous random coefficients. In particular, we present how to successfully employ the gPC expansion to this aim. Once again, to aid the discussion we introduce the following example. 

    Let us consider the stationary Allen--Cahn equation \eqref{eq:AC_random} on $D=[0,\pi] \subset \Rset$ with $g(x,\YY) = Y_1 \cos(Y_2 x)$, $Y_1$ and $Y_2$ being independent and uniformly distributed random variables, i.e.\ we consider a periodic potential with random amplitude and frequency and $N=2$. This choice is analogous to what was done in \cite{Kao2014} for the Swift--Hohenberg equation. 
    We perform the spatial discretization as described in Section \ref{sect:UQ_hom_rand} and obtain the following system in the interior of the domain $D$
	\begin{equation}\label{eq:AC_discrete_het}
		K \uu + p \uu + G(\yy) \uu  - \uu^3 = \mathbf{0}, \quad p \in \Rset, \, \yy \in \Gamma ,
	\end{equation}
	where $K\in\Rset^{m \times m}$  is the standard central finite-difference discretization matrix of the Laplacian and $G(\yy) = \text{diag}(g(x_1,\yy),\ldots,g(x_m,\yy)) \in \Rset^{m \times m}$. 
    The results reported in this section are obtained for $m=100$, $ Y_1 \sim \text{Unif}(-1,1)$, and $ Y_2 \sim \text{Unif}(-\pi/2,\pi/2)$. 
	
	  \paragraph{Bifurcation points}
    We construct a gPC approximation $p^{*}_{\text{PC},z}$ of each bifurcation value $p^*$ as  
	\begin{equation*}\label{eq:PC_bifurcation}
		p^*(\yy) \approx p^*_{\text{PC},\Lambda}(\yy) = \sum_{\boldsymbol{\alpha} \in \Lambda} \widehat{p^*_{\boldsymbol{\alpha}}} \psi_{\boldsymbol{\alpha}} (\yy)
	\end{equation*}
	with multi-index set $\Lambda$ dictated by the employed sparse-grids-based approximation as described in Appendix \ref{appendix:sg}
 %\eqref{eq:index_set}, coefficients  $\widehat{p^*_{\boldsymbol{\alpha}}}$
     and polynomials $\psi_{\boldsymbol{\alpha}}$ as described in Section \ref{sect:PC}. Since the underlying probability measure is uniform, the family of Legendre polynomials is chosen, see e.g.\ \cite{XiuKarniadakins2002:gPC}. 
    
	In Figure \ref{fig:eig_het_rand} panel (a) we display the gPC approximation of the first bifurcation value obtained with the multi-index set displayed in Figure \ref{fig:sg_pc} (a) of cardinality 25 (i.e.~$\vert \Lambda\vert = 25$).
     This requires the solution of 25 eigenvalue problems corresponding to the sparse-grids collocation points of symmetric Leja type (see e.g.\ \cite{marchi:leja}) displayed in Figure \ref{fig:sg_pc} (b). 
     %\Cquest{with total degree $z=6$ which requires the solution of 25 deterministic eigenvalue problems}. 
    Figure \ref{fig:eig_het_rand} (b) shows the decay of the root mean-square error $\lVert p^* - p^*_{\text{PC},\Lambda} \rVert_{L^2_{\rho}}$ for multi-index sets of increasing cardinality $\lvert \Lambda\rvert$ (log-log plot) 
 %\Cquest{increasing values of $z=2,\ldots,20$} 
 with respect to a reference expansion consisting of 313 terms and polynomials of total degree less or equal than 24. %\Cquest{total degree 30}. 
 The convergence behavior depends on the regularity of the function to be approximated, which is analytic in this case. An optimal choice of $\Lambda$ would lead to exponential convergence with respect to $\lvert \Lambda \rvert$, see e.g.~\cite{xiu2010:book}. Our choice might be suboptimal and further investigation with respect to this matter are out of scope here. 
 %\textcolor{red}{The plot is in semi-logarithmic scale and shows exponential decay of the error. This is expected due to the analytic dependence of the bifurcation point on $\yy$, see the proof of Proposition \ref{prop:bif_rv_hom_rand} and \cite{xiu2010:book}. }
    The root mean-square error is estimated using 10,000 samples in $\Gamma$. 
    %\Cquest{The underlying sparse-grid-approximations are based on symmetric Leja collocation points (see e.g.\ \cite{marchi:leja}), that are among the suitable family of points for uniform distributions. }

    A sufficiently accurate gPC approximation $p^*_{\text{PC},\Lambda}$ can then be sampled to estimate the pdf of the bifurcation points. In Figure \ref{fig:eig_het_rand} panel (c) we display an estimate of the pdf of the first three bifurcation points obtained with 10,000 samples. %  of $p^*_{\text{PC},6}$.  
    Note that we employ the same multi-index set %\Cquest{truncation $z=6$} 
    for all the bifurcation points, since we expect a similar dependence on $\yy$.   
 
    %As hinted above, the coefficients can be computed in a non-intrusive manner, which in this case means solving the deterministic eigenvalue problems corresponding to chosen collocation points in $\Gamma$.
	%This means that one can then solve the eigenvalue problems corresponding to the prescribed collocation points in parallel, store all the required eigenvalues and ``recombine'' them according to the methodology recalled in Appendix \ref{appendix:sg}.   

    \begin{figure}
		  \begin{minipage}[t]{0.49\textwidth}
			\centering
			\includegraphics[scale=0.45]{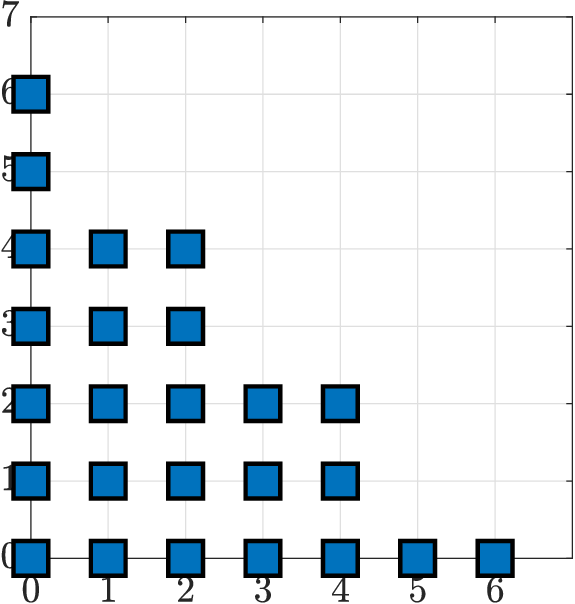}
			\subcaption{gPC multi-index set $\Lambda$.}
		\end{minipage} \hfill
        \begin{minipage}[t]{0.49\textwidth}
			\centering
			\includegraphics[scale=0.45]{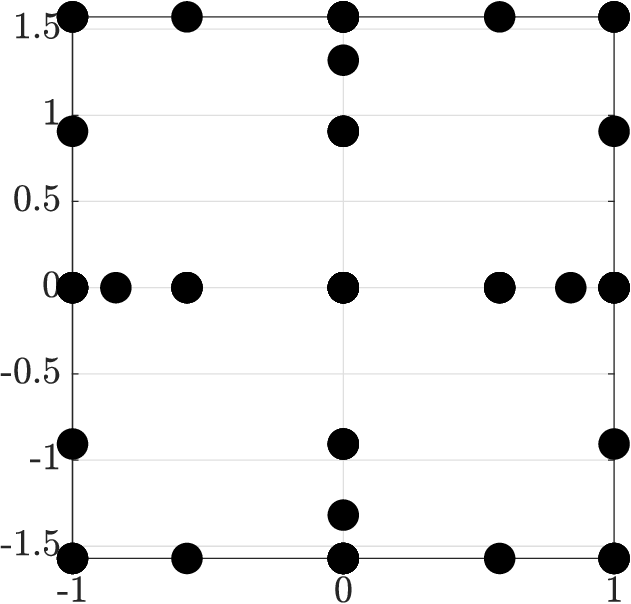}
			\subcaption{Sparse grid consisting of symmetric Leja points.}
		\end{minipage}
		\caption{Allen--Cahn equation with $g(x,\yy) = y_1 \cos(y_2 x)$, $Y_1 \sim \text{Unif}(-1,1)$, $Y_2 \sim \text{Unif}(-\pi/2,\pi/2)$. }
	\label{fig:sg_pc}
	\end{figure}
   
    \begin{figure}[tb]
        \begin{minipage}[t]{0.49\textwidth}
			\centering
			\includegraphics[scale=0.48]{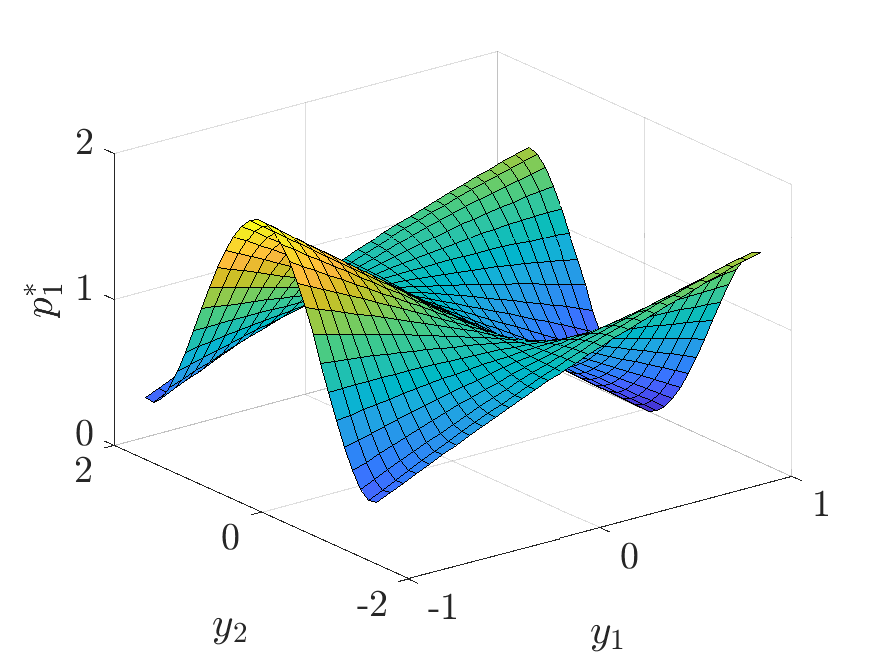}
			\subcaption{gPC approximation $p^*_{\text{PC},\Lambda}$ of $p^*_1$.}
		\end{minipage}
        \begin{minipage}[t]{0.49\textwidth}
			\centering
			\includegraphics[scale=0.45]{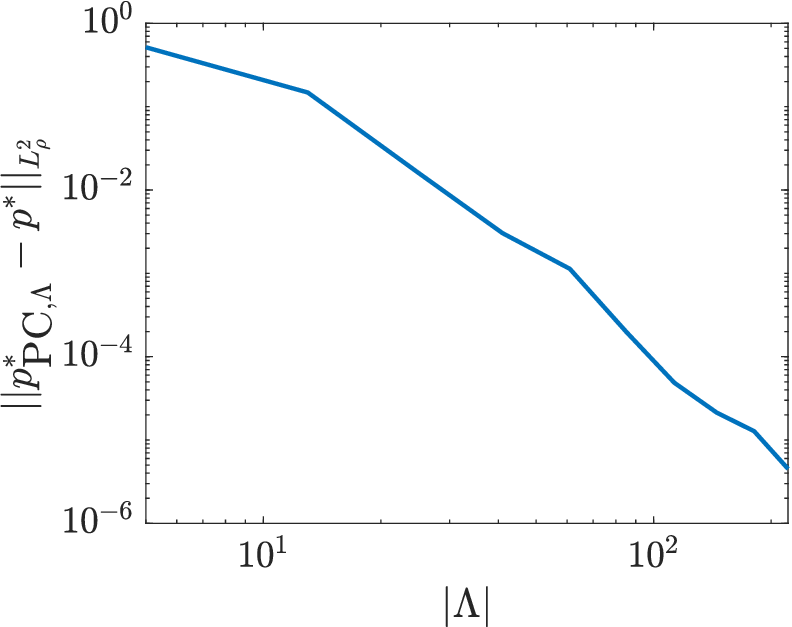}
			\subcaption{Convergence of the gPC approximation $p^*_{\text{PC},\Lambda}$ of $p^*_1$.}
		\end{minipage} \\
	    \begin{minipage}{0.5\textwidth}
			\centering
			\includegraphics[scale=0.49]{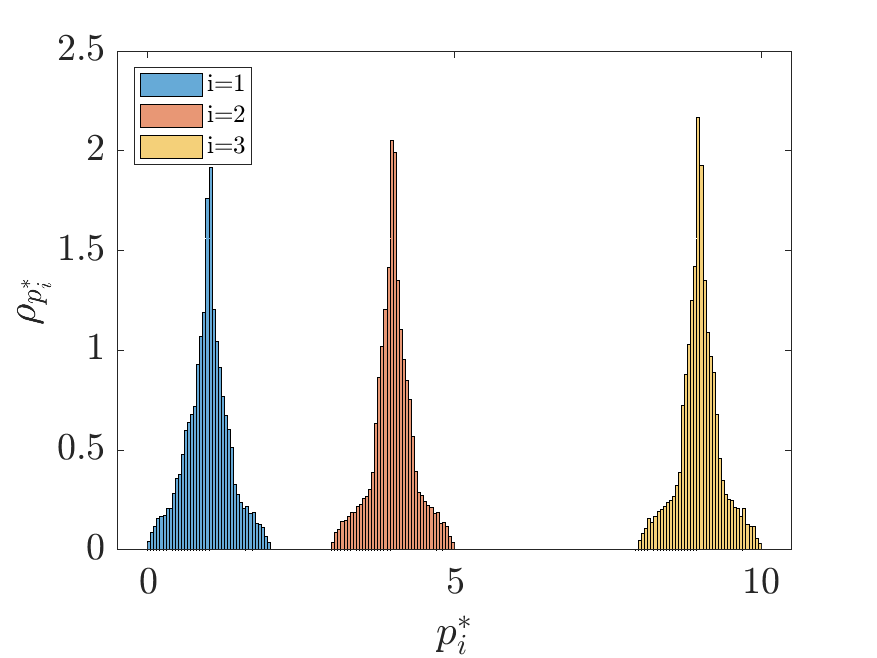}
			\subcaption{Pdf estimate of $ p_i^*(\YY)$, $i=1,2,3$.}
		\end{minipage}% \hspace{-4em}
		\begin{minipage}{0.5\textwidth}
			\centering
			\includegraphics[scale=0.49]{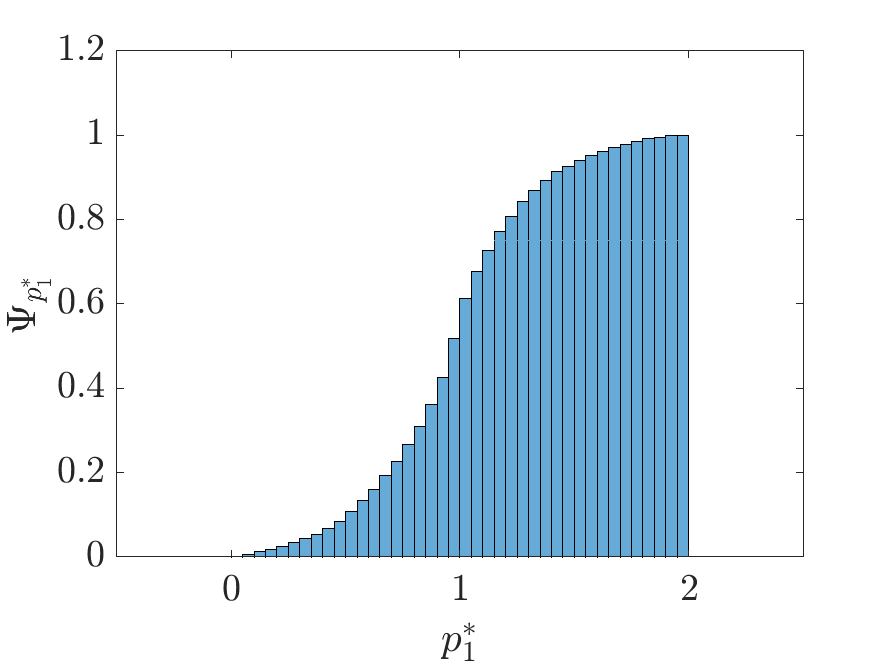}
			\subcaption{Empirical cdf of $ p_1^*(\YY)$.}
		\end{minipage}
		\caption{Allen--Cahn equation with $g(x,\yy) = y_1 \cos(y_2 x)$, $ Y_1 \sim \text{Unif}(-1,1)$, $ Y_2 \sim \text{Unif}(-\pi/2,\pi/2)$. The plots in panel (a), (c), and (d) are obtained with $\Lambda$ as in Figure \ref{fig:sg_pc} (a).} 
  %\Cquest{We display gPC approximations and statistical properties of the bifurcation points $p_i^*$}.
		\label{fig:eig_het_rand}
	\end{figure}
 
 	  %\begin{itemize}
		%$$\item \Cadd{If we do not manage to show it in the infinite dim. case, argue here that the bifurcation values are random variables. We can also argue that they have finite second-order moments in the discrete case.}
	%\end{itemize}
	%This can be done by sampling the surrogate models and applying a kernel density estimator to such samples, see e.g.\ \cite{parzen1962}. Note that pdf estimation via surrogate models shows some limitations \cite{sagiv2022}, but this investigation is out of scope here.

    \paragraph{Probability of bifurcating}
    As discussed in Section \ref{sect:UQ_hom_rand}, the probability of bifurcating at values $ p \leq \bar p$  
    %the boundary of the stability region of the trivial equilibrium can be assessed by means of the 
    is given by the cdf of the first bifurcation point, see \eqref{eq:cdf}.
    In this case, we can use the same samples of the gPC expansion \eqref{eq:PC_bifurcation} used to estimate the pdf to obtain the empirical cdf of $p_1^*$. We display it in Figure \ref{fig:eig_het_rand} panel (d).  
    
    \paragraph{Nontrivial bifurcation curves}
    The key assumption in order to use the gPC method for the approximation of bifurcation curves is that each branch can be interpreted as a two-dimensional stochastic process (cf.\ Assumption \ref{hp:u_regularity}). 
    The presentation of the gPC expansion in Section \ref{sect:PC} can then be straightforwardly extended to the case of stochastic processes 
    in $L^2_{\rho}(\Gamma;\Rset) \times L^2_{\rho}(\Gamma;L^2(D;\Rset))$.
    
    Thus, we employ the gPC expansion for each component of $\gamma_i(\YY)$ in \eqref{eq:branch_het_rand_general} and write        
    \begin{equation}\label{eq:PC_branches}
		\begin{aligned}
			r(s,\yy) & \approx r_{\text{PC}, \Lambda}(s,\yy) := \sum_{\boldsymbol{\alpha} \in \Lambda} \widehat{r}_{\boldsymbol{\alpha}}(s) \psi_{\boldsymbol{\alpha}}(\yy),  \\
			u(s,\yy) & \approx u_{\text{PC}, \Lambda}(s,\yy) := \sum_{\boldsymbol{\alpha} \in \Lambda} \widehat{u}_{\boldsymbol{\alpha}}(s) \psi_{\boldsymbol{\alpha}}(\yy)
		\end{aligned}
	\end{equation}
    with multi-index set $\Lambda$ dictated by the employed sparse-grids-based approximation as described in Appendix \ref{appendix:sg}, coefficients $\widehat{r}_{\boldsymbol{\alpha}}$ and $\widehat{u}_{\boldsymbol{\alpha}}$ that depend on the index $s \in [0,S]$ of the stochastic process, and polynomials $\psi_{\boldsymbol{\alpha}}$, as described in Section \ref{sect:PC}. 
    Note that we approximate both curve components in the same polynomial space (i.e.\ we use the same multi-index set $\Lambda$) since we expect them to have a similar dependence on $\yy$. 
        
    The computation of the expansion coefficients requires the numerical approximation of the deterministic bifurcation curves corresponding to the collocation points identified by the sparse-grid scheme (see Appendix \ref{appendix:sg}). Crucially, all such deterministic branches must share the same curve parametrization by $s$ such that the ensemble of deterministic curves can be recombined to obtain the coefficients of the gPC expansion. This is achieved by the so-called pseudo-arclength continuation algorithm that we briefly recall in Appendix \ref{appendix:num_cont}. The approximation of each curve is obtained stepping on the arclength parameter, and, once we fix this step to be the same for each branch realization, we ensure that all curves are provided within the same reference framework. As in the spatially-homogeneous case, we employ the numerical continuation software Continuation Core and Toolboxes (COCO) \cite{dankowicz2013} (version October 26, 2023) available at \cite{coco:sourceforge} for the approximation of each realization. 

    In Figure \ref{fig:bif_samples_het_rand} we display realizations of the first bifurcation curve for three scalar observables associated to the $u$-component of the branch. The curves in black are the corresponding mean bifurcation curves \eqref{eq:mean_bif_curve} which are given by the first coefficient of the gPC expansion (see Section \ref{sect:PC}). We display the mean bifurcation diagram in Figure \ref{fig:mean_sol_het_rand} panel (a) and the solution along the mean bifurcation branch in panel (b). In particular, the trivial solution is obtained for $p=\mathbb{E}[p_1^*(\YY)]$ and the curve with the largest maximum value corresponds to the largest value of $p$ considered (i.e., the point where we stop the numerical continuation method for the computation of the reference branch).
    
    In Figure \ref{fig:conv_het_rand} we show the convergence behavior of the gPC expansion \eqref{eq:PC_branches} at value of arclength $s=5$
    with respect to sets $\Lambda$ with increasing cardinality. 
    %\Cquest{increasing values of the total degree $z$.} 
    The reference solution constists of 313 terms with total degree smaller or equal than 24.
    %\Cquest{has total degree 30}. 
    The plot is in log-log scale and shows a decay similar to the one obtained for the bifurcation points. The root mean-square error is obtained using 10,000 samples in $\Gamma$. The underlying sparse-grid-approximations are based on symmetric Leja collocation points, see e.g.~\cite{marchi:leja}.

    %Thus, the PC approximation of $\Gamma(\yy)$ boils down to the approximation of a stochastic process $u(s,\yy):[0,+\infty) \times \Gamma \rightarrow H_0^1(D;\Rset)$ assumed to belongs to the Bochner space $L^2_{\rho}(\Xi;H^1_0(D;\Rset))$. We write 

	%As observed in Section \ref{sect:PC}, the first coefficient of the PC expansion is the expected value of the considered quantity, i.e.\ we can formally write
	%	\[
		%(\mathbb{E}[p](s),\mathbb{E}[u](s)) = (\widehat{p}_{\bf 0}(s),\widehat{u}_{\bf 0}(s))
		%\] 
		%and define the following curve
		%\[
		%\Gamma^+_{e} = \{ (\widehat{p}_{\bf 0}(s),\widehat{u}_{\bf 0}(s)) \, \vert \, s \in [0,+\infty), (\widehat{p}_{\bf 0}(0),\widehat{u}_{\bf 0}(0)) = (\mathbb{E}[p^*],0) \}.
		%\]	
  
    \begin{figure}[tb]
		\begin{minipage}{0.32\textwidth}
			\centering
			\includegraphics[width = 0.9\textwidth]{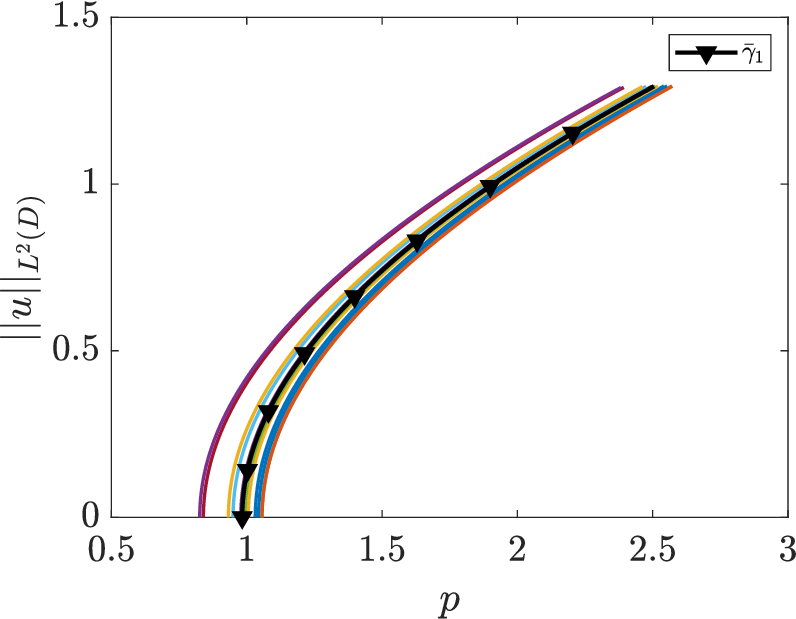}
			\subcaption{$\mcU = \lVert u \rVert_{L^2(D)}$.}
		\end{minipage}% \hspace{-4em}
        \begin{minipage}{0.32\textwidth}
			\centering
			\includegraphics[width = 0.9\textwidth]{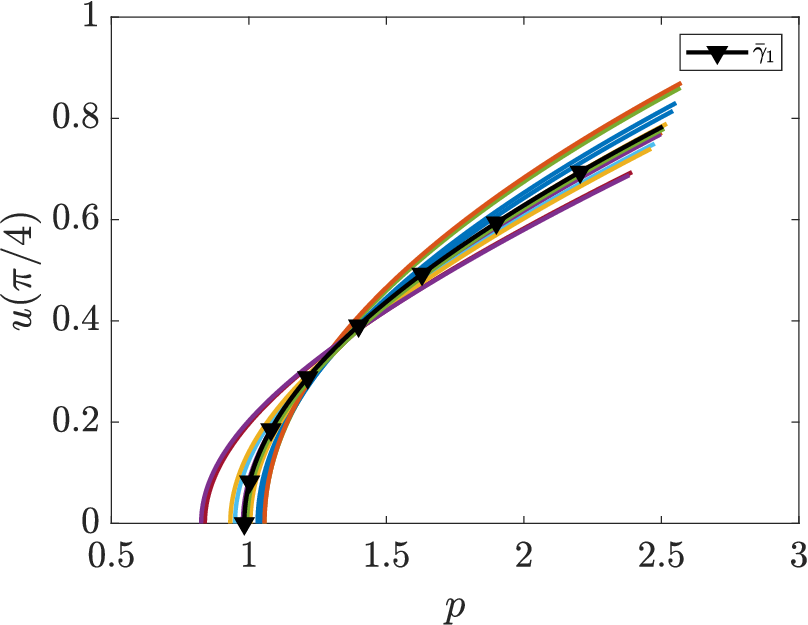}
			\subcaption{$\mcU = u(\pi/4)$.}
		\end{minipage}% \hspace{-4em}
		\begin{minipage}{0.32\textwidth}
			\centering
			\includegraphics[width = 0.9\textwidth]{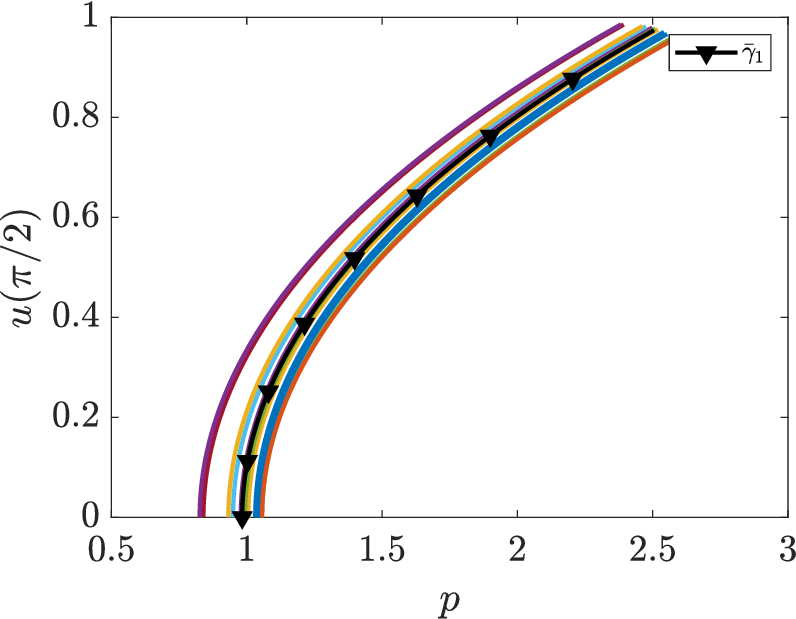}
			\subcaption{$\mcU = u(\pi/2)$.}
		\end{minipage}
		\caption{Allen--Cahn equation with $g(x,\yy) = y_1 \cos(y_2 x)$, $Y_1 \sim \text{Unif}(-1,1)$, $Y_2 \sim \text{Unif}(-\pi/2,\pi/2)$. Realizations of the first bifurcation branch for three different scalar observables $\mcU$. We mark in black the corresponding mean bifurcation curve \eqref{eq:mean_bif_curve}.}
		\label{fig:bif_samples_het_rand}
	\end{figure}

    \begin{figure}
		\begin{minipage}[t]{0.48\textwidth}
			\centering
			\includegraphics[scale=0.45]{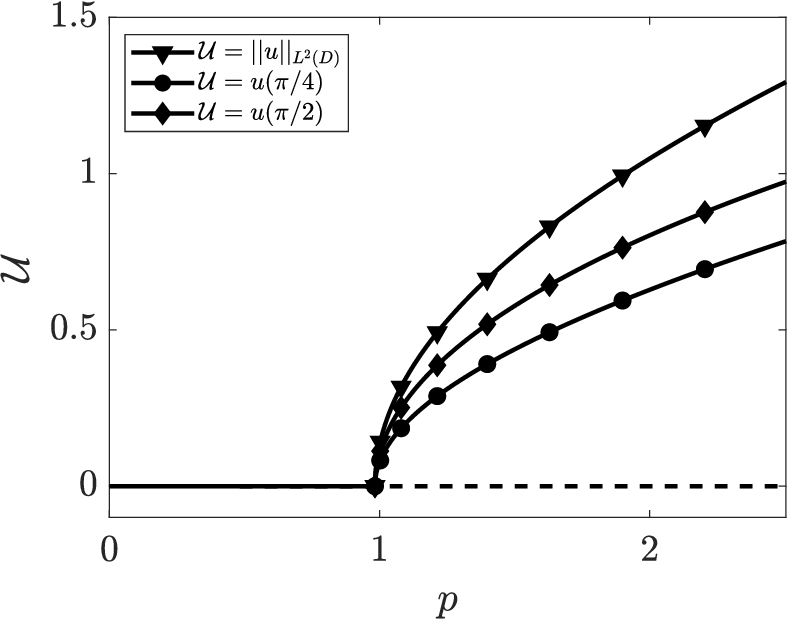}
			\subcaption{Mean bifurcation diagram for the observables $\mcU$.}
		\end{minipage}
        \begin{minipage}[t]{0.48\textwidth}
			\centering
			\includegraphics[scale=0.45]{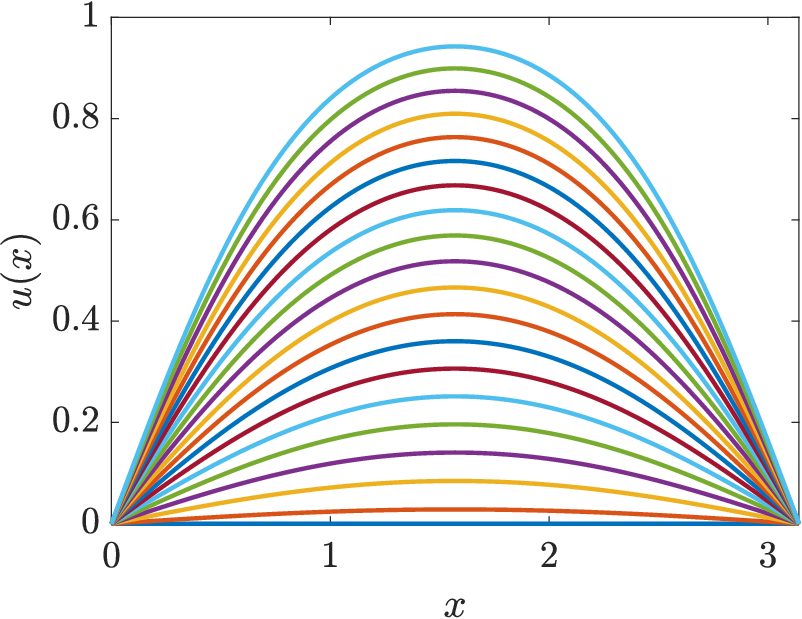}
			\subcaption{Solution along the mean bifurcation branch. The different curves correspond to increasing values of the bifurcation parameter $p$.}
		\end{minipage}
		\caption{Allen--Cahn equation with $g(x,\yy) = y_1 \cos(y_2 x)$, $Y_1 \sim \text{Unif}(-1,1)$, $Y_2 \sim \text{Unif}(-\pi/2,\pi/2)$. }
	\label{fig:mean_sol_het_rand}
	\end{figure}

    \begin{figure}
        \centering
        \includegraphics[scale=0.5]{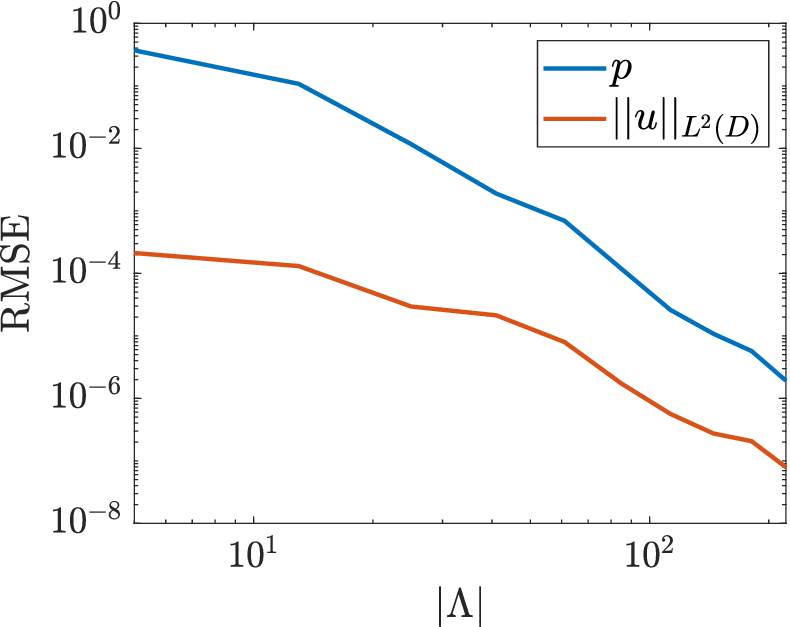}
        \caption{Allen--Cahn equation with $g(x,\yy) = y_1 \cos(y_2 x)$, $ Y_1 \sim \text{Unif}(-1,1)$, $Y_2 \sim \text{Unif}(-\pi/2,\pi/2)$. Convergence of the gPC approximation of the first bifurcation branch at $s = 5$.}
        \label{fig:conv_het_rand}
    \end{figure}

%%%%%%%%%%%%%%%%%%%%%%%%%%%%%%%%%%%%%%%%%%%%%%%%%%%%%%%%%%%%%%%%%%%%%%%%%%%%%%%%%%%%%%%%%%%%%%%%%%%%%%%%%%%%%%%%%%%%%%%%%%%%%%%%%%%%%%%%%%%%%

\section{Discussion, conclusion and outlook}\label{sect:conclusion}
In this work we considered the Allen--Cahn equation with random coefficients. We discussed the forward UQ analysis in the presence of bifurcations, which   
%
%Specifically, we modeled the uncertainty by a random field involving hyper-parameters and discretized by a finite number of random variables. In the considered setting, one of the hyper-parameters, namely the expected value of the random field, leads to supercritical pitchfork bifurcations. 
%
%The UQ analysis of PDEs in presence of bifurcation 
poses the challenge of dealing with non-uniqueness. Indeed, a small variation of the bifurcation parameter in a neighborhood of a bifurcation point leads to a change in the number of the steady states and multiple solutions correspond to the same parameter value. 

We tackled this challenge in two main steps. 
First, we translated the Allen--Cahn equation with random coefficients to a probability space associated with a finite number of random variables considered in the model.
%\Cquest{parametric problem, with the parameter space being the Cartesian product of the image spaces of the random variables involved in the model.} %Note that these parameters are different from (and should not be confused with) the bifurcation parameter and the other hyper-parameters in the random field model. 
In this setting, the classic bifurcation analysis of deterministic dynamical systems is employed pathwise for the stochastic problem (i.e.\ for each realization of it) to conclude about the occurrence of supercritical pitchfork bifurcations in the random dynamical system. Moreover, bifurcation points are shown to be random variables and bifurcation curves to be random curves. 
Second, we proposed a numerical methodology to assess the uncertainty affecting bifurcation points and corresponding branches that relies on the gPC expansion. Crucially, the approach takes advantage of the previous bifurcation analysis to obtain gPC approximations of bifurcation points and bifurcation curves separately. For the computation of the gPC coefficients we rely on sparse-grids stochastic collocation that results in solving a set of deterministic eigenvalue problems and approximating by numerical continuation methods the corresponding nontrivial branches of equilibria. 

The proposed gPC approach returns surrogate models that can be explored to perform efficiently the UQ analysis. For example, the probability of bifurcating can be computed by sampling the gPC model of the first bifurcation point. Furthermore, the gPC expansion can be used to extract information on the branches, such as tangent and curvature along them that in some application might be informative of some features of the dynamical system (think for example at the velocity of approaching a certain bifurcation point).  It is therefore of great interest for us to extend this work towards real-world problems and tackle possible problem-specific questions.
We also believe that our approach has a large potential for the study of random dynamical systems with more involved bifurcation diagrams, for example \cite{Morgan2014,Uecker2020}. A natural future research direction is the investigation of more complicated models, e.g.\ when secondary bifurcations arise at nontrivial equilibria, or when the randomness alters the bifurcating behavior. The surrogate models could be used to check the bifurcation conditions in a computationally efficient manner. Moreover, we plan to consider different random field models to account for different random spatially-heterogeneous effects in the PDE model and identify possible hyper-parameters leading to bifurcations. % This type of analysis could be integrated within the research field on hyper-parameters estimation.  

%%%%%%%%%%%%%%%%%%%%%%%%%%%%%%%%%%%%%%%%%%%%%%%%%%%%%%%%%%%%%%%%%%%%%%%%%%%%%%%%%%%%%%%%%%%%%%%%%%%%%%%%%%%%%%%%%%%%%%%%%%%%%%%%%%%%%%%%%%%%%%

\section*{Acknowledgments}
We thank the reviewers whose comments and suggestions helped to improve and clarify this manuscript.
Chiara Piazzola acknowledges the support of the Alexander von Humboldt foundation. Christian Kuehn would like to thank the VolkswagenStiftung for support via a Lichtenberg Professorship. Chiara Piazzola is a member of the Gruppo Nazionale Calcolo Scientifico-Istituto Nazionale di Alta Matematica (GNCS-INdAM).

%%%%%%%%%%%%%%%%%%%%%%%%%%%%%%%%%%%%%%%%%%%%%%%%%%%%%%%%%%%%%%%%%%%%%%%%%%%%%%%%%%%%%%%%%%%%%%%%%%%%%%%%%%%%%%%%%%%%%%%%%%%%%%%%%%%%%%%%%%%%%%

\appendix
\section{Numerical continuation} \label{appendix:num_cont}

Branches of equilibria (we denote them simply by $\gamma$ in the following) are implicitly defined by a nonlinear equation $F(p,\uu)=0$ that depends on the bifurcation parameter $p$ and has to be understood as the spatially-discretized version of the stationary Allen--Cahn equation, cf.\ \eqref{eq:AC_discrete_hom}, \eqref{eq:AC_discrete_het}, i.e.\ $\uu \in \Rset^m$. Roughly speaking, numerical continuation is a method to solve such parametric equations, where the solution obtained at a certain value of $p$ is then used as an initial guess for computing the solution at another value of the parameter. In other words, a solution is extended (or equivalently continued) from a value of the parameter to another one, see e.g.\ \cite{kuznetsov2013,uecker2021}. 

In this work, instead of stepping on the parameter $p$ (usually called natural continuation, see \cite{kuznetsov2013}), we resort to another classic version of numerical continuation, the so-called pseudo-arclength continuation, which is based on the arclength parametrization of a bifurcation curve. In practice, this means solving the following system of equations,
\begin{equation}\label{eq:num_cont}
    H(p,\uu) = \begin{pmatrix}
	F(p,\uu) \\ \mu(p,\uu,s)
	\end{pmatrix} = 0,  %\quad G: \Rset \times \Rset^M \rightarrow \Rset^M,
\end{equation}
where the nonlinear problem $F(p,\uu)=0$ is coupled with an additional equation that enforces the arclength parametrization. In details, given a starting point $\gamma^{(0)} = (p(s_0),\uu(s_0))$ at $s = s_0$ and the corresponding tangent vector $\dot{\gamma}^{(0)} = (\dot{p}(s),\dot \uu(s)) \vert_{s=s_0}$, the constraint is given by 
\[
\mu(p,\uu,s) =  (1-\xi) \; \dot p(s_0) \; \big( p(s)-p(s_0) \big)  + \xi \; \langle \dot \uu(s_0), \uu(s)-\uu(s_0) \rangle  - (s-s_0) . 
\]
Note that $\mu(p,\uu,s)=0$
defines a hyperplane perpendicular to $\dot{\gamma}^{(0)}$ at distance $\Delta s = s -s_0$ from $\gamma^{(0)}$ (see Figure \ref{fig:continuation} for an illustration) in the inner product
\[
\bigg \langle 
\begin{pmatrix}
    a \\ b
\end{pmatrix}, \begin{pmatrix}
    c \\ d
\end{pmatrix} \bigg \rangle_{\xi} = (1-\xi) \, ac + \xi \langle b,d \rangle, \quad \, \text{with } \, \xi \in (0,1),  
\]
with $\langle \cdot, \cdot \rangle$ being the Euclidean inner product in $\Rset^m$. 
The system \eqref{eq:num_cont} is then usually solved by a predictor-corrector strategy. 
Starting with $\gamma^{(0)}$, the prediction $\widetilde{\gamma}^{(1)}$ after one step of length $\Delta s$ is obtained in the tangent space of the curve as follows
\[
\widetilde{\gamma}^{(1)} =\gamma^{(0)} + \Delta s \dot{\gamma}^{(0)}.
\]
The correction step is implemented by the Newton's method 
\[
\gamma^{(1)} = \widetilde{\gamma}^{(1)} - J_H \big (\widetilde{\gamma}^{(1)} \big)^{-1} H \big (\widetilde{\gamma}^{(1)} \big )
\]
with $J_H \big (\widetilde{\gamma}^{(1)} \big )$ being the Jacobian matrix of $H$ at the point $\widetilde{\gamma}^{(1)}$. 
\begin{figure}[htb]
	\centering
	\resizebox{0.4\textwidth}{!}{\input{continuation.tikz}}
	\caption{Illustration of the pseudo-arclength continuation method. For visualization purposes we display $\mcU$ on the vertical axis which is a scalar observable associated to $\uu$.}
	\label{fig:continuation}
\end{figure}

\section{Sparse-grids-based computation of the gPC coefficients}\label{appendix:sg}
    In this appendix we briefly recall the basics of sparse grids and summarize the approach implemented in the Sparse Grids Matlab kit \cite{piazzola2022:sparse.grids,sgmk:github} for the computation of the gPC approximation \eqref{eq:PC_truncated}. 
    
    A sparse-grids-based approximation $f_{\text{SG}}$ of an $N$-variate function $f: \Gamma \rightarrow \Rset$ can be written as a linear combination of Lagrangian interpolants on a set of suitably chosen tensor grids on $\Gamma$, each of them consisting of a small number of points. In formulas, we have
    \begin{equation}\label{eq:sg}
    f(\yy)  \approx f_{\text{SG}}(\yy) := \sum_{\ii \in \mcI} c_{\ii} \mcL_{\ii}(\yy), \quad \text{with} \quad c_{\ii}: = \sum_{\substack{{\jj} \in \{0,1\}^N \\ \ii+{\jj} \in \mcI}} (-1)^{\lVert {\jj} \rVert_1},
    \end{equation}
    where $\mcI \subset \Nset^N_+$ is a multi-index set and $\mcL_{\ii}(\yy)$ is a tensor Lagrangian interpolant of $f$ built over the tensor grid $\mathcal{T}_{\ii}$ identified by the multi-index $\ii$. 
    In order to define the tensor grid $\mathcal{T}_{\ii}$, we introduce the following univariate set of $K_n \in \mathbb{N}_+$ collocation points along $y_n$ 
	\[
		\mathcal{T}_{i_n} = \left\{y_{n,m(i_n)}^{(j_n)}: j_n=1, \ldots, K_n \right\} \quad \text{ for } n=1,\ldots,N,
	\]
    where the function $m$ is the so-called ``level-to-knots'' function that associates to each level $i_n$ a number of points:
	\begin{equation*}
		m:\mathbb{\mathbb{N}}_+ \rightarrow \mathbb{N}_+ \text{ such that } m(i_n) = K_n.
	\end{equation*}
    Then, the tensor grid is given by 
    \[
        \mathcal{T}_{\ii} = \left\{\yy_{m(\ii)}^{(\jj)}\right\}_{\jj \leq m(\ii)},
    	\quad  \text{ with } \quad
    	\yy_{m(\ii)}^{(\jj)} = \left[y_{1,m(i_1)}^{(j_1)}, \ldots, y_{N,m(i_N)}^{(j_N)}\right]
    	\text{ and } \jj \in \mathbb{N}^N_+,
    \]
    where $m(\ii) = \left[m(i_1),\,m(i_2),\ldots,m(i_N) \right]$ and $\jj \leq m(\ii)$ means that $j_n \leq m(i_n)$ for every $n = 1,\ldots,N$.
    The choice of collocation points is based on the distribution of each parameter $y_n$ and the function $m$ is then chosen accordingly for efficiency reasons, see \cite{piazzola2022:sparse.grids} and references therein for a thorough discussion. 
    The choice of the multi-index set depends on the regularity of $f$, see e.g.~\cite{bungartz2004,Chkifa2014}. One common option is the following
    \begin{equation}\label{eq:I_sparse}
		\mathcal{I} = \{ \ii \in \mathbb{N}^N_+ : \sum_{n=1}^N (i_n-1) \leq w \}  
	\end{equation}
	for some $w \in \Nset$. We refer to \cite{piazzola.eal:user.manual} for an overview on the multiple other options that are available in the Sparse Grids Matlab Kit.
 
    Then, from such sparse-grids-based approximation the gPC equivalent expansion can be made by reformulating each Lagrangian interpolant $\mcL_{\ii}$ in terms of the appropriate family of orthogonal polynomials and then summing up all the terms. For the detailed algorithm we refer to \cite{formaggia2013}. Clearly, the gPC multi-index set $\Lambda$ in \eqref{eq:PC_truncated} is completely determined by $\mathcal{I}$ and the function $m$, see \cite{back2011}. It holds
    \[
    \Lambda =\{\aalpha \in \Nset^N : \aalpha \leq m(\ii)-\bf{1} \, \text{with } \ii \in \mathcal{I} \}. 
    \]        
    For the example in Section \ref{sect:UQ_het_rand} we consider $\mathcal{I}$ as in \eqref{eq:I_sparse} with $w=3$. The corresponding gPC multi-index set is displayed in Figure \ref{fig:sg_pc}, panel (a). Further, we use symmetric Leja points that are an appropriate choice for uniform random variables together with $m(i_n)=2i_n-1$. The resulting sparse grid consisting of 25 points and is shown in Figure \ref{fig:sg_pc}, panel (b).

\bibliographystyle{abbrvurl}
\bibliography{biblio}

\end{document}

%% file: continuation.tikz
\begin{tikzpicture}[>=angle 45]
	\tikzstyle{every node}=[font=\fontsize{10}{10}\selectfont]
	\begin{axis}[
		% x=1cm,y=1cm,
		axis lines=middle,
		axis line style=very thin,
		xmin=0,
		xmax=2,
		ymin=0,
		ymax=1.5,
		xlabel=$p$,
		ylabel=$\mathcal{U}$,
		ticks = none]
		% \clip(-0.1,-0.1) rectangle (0.8,0.5);
		\draw[line width=1.5pt, smooth,samples=20,domain=0.4:1.5] plot(\x,{0.8 - (1 -\x)^2});
		\draw [->,line width=1pt] (0.7,0.71) -- (1.2,1);
		\draw [line width=1pt,dash pattern=on 1pt off 1pt] (1.2,1)-- (1.34,0.69);
		%\begin{scriptsize}
		\draw [fill=black] (0.7,0.71) circle (1.5pt);
		\draw[color=black,font=\tiny] (0.7,0.8) node {$\gamma^{(j)}$};
		\draw [fill=black] (1.2,1) circle (1.5pt);
		\draw[color=black,font=\tiny] (1.2,1.05) node {$\widetilde{\gamma}^{(j+1)}$};
		\draw [fill=black] (1.34,0.69) circle (1.5pt);
		\draw[color=black,font=\tiny] (1.5,0.71) node {$\gamma^{(j+1)}$};
		%\end{scriptsize}
		\end{axis}
\end{tikzpicture}